\documentclass[pdflatex,sn-mathphys-num]{sn-jnl}

\usepackage{amsmath,amssymb,amsfonts}
\usepackage{amsthm}
\usepackage{mathrsfs}
\usepackage{xcolor}

\makeatletter
\newtheoremstyle{thmstyleone}%
  {18pt plus2pt minus1pt}{18pt plus2pt minus1pt}%
  {\small\itshape}{0pt}{\small\bfseries}{}%
  {.5em}{\thmname{#1}\thmnumber{\@ifnotempty{#1}{ }\@upn{#2}}%
         \thmnote{ {\the\thm@notefont(#3)}}}
\newtheoremstyle{thmstyletwo}%
  {18pt plus2pt minus1pt}{18pt plus2pt minus1pt}%
  {\small\normalfont}{0pt}{\small\itshape}{}%
  {.5em}{\thmname{#1}\thmnumber{\@ifnotempty{#1}{ }{#2}}%
         \thmnote{ {\the\thm@notefont(#3)}}}
\newtheoremstyle{thmstylethree}%
  {18pt plus2pt minus1pt}{18pt plus2pt minus1pt}%
  {\small\normalfont}{0pt}{\small\bfseries}{}%
  {.5em}{\thmname{#1}\thmnumber{\@ifnotempty{#1}{ }\@upn{#2}}%
         \thmnote{ {\the\thm@notefont(#3)}}}
\makeatother

\AtBeginDocument{\newgeometry{left=20mm,right=20mm,top=22mm,bottom=25mm}}

\theoremstyle{thmstyleone}
\newtheorem*{utheorem}{Theorem}
\newtheorem*{uproposition}{Proposition}

\theoremstyle{thmstyleone}
\newtheorem{theorem}{Theorem}
\newtheorem{proposition}[theorem]{Proposition}
\newtheorem{lemma}[theorem]{Lemma}

\theoremstyle{thmstyletwo}
\newtheorem{example}{Example}
\newtheorem{remark}{Remark}
\theoremstyle{thmstylethree}

\newtheorem{assumption}{Assumption}

\newcommand{\DD}{\mathrm{D}}

\begin{document}

\def\hbr{\hfil\break}
\newcommand\F{\mbox{I\kern-2pt F}}
\newcommand\cA{{\cal A}}
\newcommand\cE{{\cal E}}
\newcommand\cC{{\cal C}}
\newcommand\cF{{\cal F}}
\newcommand\cG{{\cal G}}
\newcommand\cH{{\cal H}}
\newcommand\cI{{\cal I}}
\newcommand\cK{{\cal K}}
\newcommand\cL{{\cal L}}
\newcommand\cB{{\cal B}}
\newcommand\cN{{\cal N}}
\newcommand\cM{{\cal M}}
\newcommand\cX{{\cal X}}
\newcommand\cD{{\cal D}}
\newcommand\cO{{\cal O}}
\newcommand\cR{{\cal R}}
\newcommand\cP{{\cal P}}
\newcommand\cQ{{\cal Q}}
\newcommand\cS{{\cal S}}
\newcommand\cT{{\cal T}}
\newcommand\cU{{\mathcal{U}}}
\newcommand\cV{{\cal V}}
\newcommand\cY{{\cal Y}}
\newcommand\cZ{{\cal Z}}
\newcommand\uv{{\underline v}}
\newcommand\uw{{\underline w}}
\newcommand\up{{\underline p}}
\newcommand\uV{{\underline V}}
\newcommand\uW{{\underline W}}
\def\bbr{{\mathbb R}}
\def\bbn{{\mathbb N}}
\def\bbc{{\mathbb C}}
\def\bbf{{\mathbb F}}
\def\bbg{{\mathbb G}}
\def\bbz{{\mathbb Z}}
\def\bbb{{\mathbb B}}

\def\R{{\mathbb R}}
\def\N{{\mathbb N}}
\def\bP{{\mathbb P}}
\newcommand\E{{\mathbb{E}}}
\newcommand\1{{\bf 1}}
\newcommand{\wh}{\widehat}

\newcommand\Lip{{\operatorname{Lip}}}

\newcommand\beq{\begin{equation}}
\newcommand\eeq{\end{equation}}
\newcommand\bea{\begin{eqnarray}}
\newcommand\eea{\end{eqnarray}}
\newcommand\bean{\begin{eqnarray*}}
\newcommand\eean{\end{eqnarray*}}

\newcommand{\Cloc}[1]{C^{#1}_{\mathrm{loc}}}                    
\makeatletter
\newcommand{\weighted@sup}[1]{%
  \def\weighted@order{#1}%
  \def\weighted@zero{0}%
  \ifx\weighted@order\weighted@zero\else^{#1}\fi
}
\newcommand{\weighted@bdsup}[1]{%
  \def\weighted@order{#1}%
  \def\weighted@zero{0}%
  \ifx\weighted@order\weighted@zero\else^{#1,\mathrm{bd}}\fi
}
\newcommand{\Cwa}[2][\alpha]{C\weighted@sup{#2}_{w_{#1}}}       
\newcommand{\Cwainf}[2][\alpha]{C\weighted@sup{#2}_{w_{#1},\infty}} 
\newcommand{\Cwbd}[1]{C\weighted@bdsup{#1}_{w_2}}               
\newcommand{\Cwbdinf}[1]{C\weighted@bdsup{#1}_{w_2,\infty}}     
\makeatother

\title[Weighted derivative estimates]{Weighted derivative estimates for Kolmogorov equations}


\author*[1,2,3]{\fnm{Artur} \sur{Sidorenko}}\email{sidorenkoap@my.msu.ru}
\author*[1,2,3]{\fnm{Vasilii} \sur{Kolokoltsov}}\email{kolokoltsov59@mail.ru}

\affil[1]{\orgname{HSE University}, \orgaddress{ \city{Moscow}, \country{Russia}}}

\affil[2]{\orgname{Lomonosov Moscow State University}, \orgaddress{ \city{Moscow}, \country{Russia}}}

\affil[3]{\orgname{Vega Institute Foundation}, \orgaddress{ \city{Moscow}, \country{Russia}}}


\abstract{
For diffusions whose coefficients are smooth, with at most linear growth and bounded derivatives of higher order, we establish explicit pointwise bounds on the spatial derivatives of the classical solution to the Cauchy problem with polynomially growing initial data. We also obtain explicit growth estimates for the derivatives of the SDE with respect to the initial data. The proofs combine Lyapunov functions and Markov semigroups on weighted spaces with Kunita's theory of the differentiable dependence of SDE solutions on the initial data.
}



\maketitle

\section{Introduction}

Let $b \colon \R^d \to \R^d$ and $\sigma \colon \R^d \to \R^{d\times m}$. We consider the Cauchy problem
\begin{equation}
    \label{eq:Cauchy}
    \left\{
    \begin{aligned}
        \partial_t u(t,x) &= \cL u(t,x),\\
        u(0,x) &= f(x),
    \end{aligned}
    \right.
\end{equation}
where
\begin{equation}\label{eq:generator-def}
\cL f(x)=b(x)\cdot \DD f(x)+\frac12\operatorname{tr}\!\bigl( a(x) \DD^2 f(x)\bigr),
\qquad a(x)=\sigma(x)\sigma(x)^\top,
\end{equation}
on $\R^d$. Let $W$ be a standard $m$-dimensional Brownian motion on a filtered probability space. The diffusion $X_t = X_t(x)$ is the strong solution to the stochastic differential equation
\begin{equation}
    \label{eq:SDE}
    dX_t = b(X_t)\,dt + \sigma(X_t)\,dW_t, \qquad X_0=x \in \R^d.
\end{equation}
We identify $\R^{d\times m}$ with $(\R^d)^m$ and write
\[
\sigma=(\sigma_1,\dots,\sigma_m),
\]
where $\sigma_\ell:\R^d\to\R^d$ denotes the $\ell$-th column of $\sigma$.
The operator $\cL$ is the generator of $X_t(x)$, and its Markov semigroup is denoted by
\[
P_tf(x):=\E_x[f(X_t(x))].
\]
Under suitable conditions, the solution to the Cauchy problem admits the probabilistic representation $u(t,x)=P_tf(x)$. We estimate spatial derivatives of both the solution $x\mapsto X_t(x)$ of the SDE and the semigroup $P_tf$.

Classical theory for bounded coefficients is by now standard: see \cite[Ch.~IV]{ladyzhenskaia1968linear}, \cite[Ch.~7--8]{krylov1996lectures}, and \cite[Chs.~4 and 6]{lieberman1996second}.
For derivatives of the semigroup $P_tf(x)=\E_x[f(X_t)]$ when the coefficients are unbounded, two approaches are especially relevant here. The first is the parametrix method: one constructs the transition density as a perturbation series around a reference kernel. Kolokoltsov \cite{Kolokoltsov2000Stable} used this construction for stable-like jump-diffusions; Konakov, Menozzi and Molchanov \cite{KonakovMenozziMolchanov2010} treated the degenerate Brownian case; Menozzi, Pesce and Zhang \cite{MenozziPesceZhang2021} obtained bounds on transition densities $\DD p$ and $\DD^2 p$ in the non-degenerate Brownian case.
The second approach uses Malliavin calculus. Priola \cite{Priola2006} obtained formulas for $\DD^k P_t f$ at $k=1,2,3$ of the Ornstein--Uhlenbeck SDE $dX_t=(AX_t+F(X_t))\,dt+Q^{1/2}\,dW_t$. Da Prato and Priola \cite{DaPratoPriola2018} proved the bound $\sqrt{t}\,|\DD_xP_t\varphi(x)|\le c\,\nu^{-1/2}(1+|x|^k)\,\|\varphi\|_\infty$ for nondegenerate SDEs with $\sigma\sigma^\top\ge\nu I$, and gave a one-dimensional counterexample with $C^1$ drift of sublinear growth but unbounded derivative for which
\[
\sup_{t\in(0,1]}\ \sup_{\|f\|_\infty\le 1}\sqrt{t}\,\|\DD_xP_tf\|_\infty=\infty.
\]
This counterexample is the reason to work in weighted spaces.
Palczewski \cite{Palczewski2025} proved derivative growth bounds for solutions of uniformly elliptic Kolmogorov equations with linearly growing coefficients.

At first, we provide explicit moment bounds for the spatial derivatives of the solutions to the SDE. These estimates treat the case of both degenerate and non-degenerate diffusions. Using these bounds, we give growth estimates for the derivatives of the Cauchy problem for initial data of polynomial growth. Two cases are treated. In the first, $b$ and $\sigma$ are of at most linear growth and their $j$-th derivatives for $j\ge 2$ decay at infinity like $|x|^{1-j}$; this covers geometric Brownian motion and the Ornstein--Uhlenbeck process, and the pointwise estimate takes the form $\|\DD^k u(t,x)\|$. In semigroup terms, this regime gives a quasi-contraction on the weighted spaces.  In the second case, the higher derivatives are merely uniformly bounded and the Cauchy problem is treated for initial conditions being functions whose first two derivatives are bounded and vanish at infinity. Here the estimate on the second derivative carries an extra factor $1+Mt^{1/2}$, which we do not show to be sharp. 

The proofs combine three ingredients. We use Lyapunov functions and Markov semigroups on weighted spaces of smooth functions developed in \cite[\S 5.2]{Kolokoltsov2010Nonlinear} and  \cite[\S 5.6]{Kolokoltsov2011Markov}, see also \cite{CuchieroFilipovicTeichmann2020,CuchieroMoellmannTeichmann2026}. To differentiate the spatial map $x\mapsto X_t(x,\omega)$, we use Kunita's theory of the differentiable dependence of SDE solutions on the initial data \cite{Kunita1984SaintFlour, Kunita1990, Kunita2019}. The chain rule for the tensor derivatives $\DD^k(b\circ X_t)$ and $\DD^k(\sigma_\ell\circ X_t)$ is given by the multivariate Fa\`a di Bruno formula, Proposition 3.1 of \cite{Licht2024}.

The paper is organized as follows. In Section~\ref{sec:setting}, we fix the setting and state the main results. In Section~\ref{sec:inv}, we give the Lyapunov estimates used for the semigroup. In Section~\ref{sec:tensor}, we introduce the tensor notation and the derivative of the solutions to SDEs. In Section~\ref{sec:derivative} we prove the estimates for these derivatives. Finally, Section~\ref{sec:pde_appl} transfers these bounds to the PDE.

\section{Setting and main results} \label{sec:setting}

Throughout this paper, $c$ denotes a finite positive constant whose value may change from line to line. We denote by $C^k(\R^d)$ the space of bounded functions with bounded continuous derivatives up to order $k$, equipped with the norm $\|g\|_{C^k}:=\max_{0\le j\le k}\sup_{x\in\R^d}\|\DD^j g(x)\|$, and by $C_\infty(\R^d)$ the space of continuous functions vanishing at infinity. We
use $\Cloc{k}(\R^d)$ for $k$-times continuously differentiable functions with no growth or boundedness conditions. Spatial derivatives are written in tensor form as $\DD^k f(x)$, with the convention $\DD^0 f = f$ and $\DD^1 f = \nabla f$; the tensor norm $\|\cdot\|$ is standard and recalled in Section~\ref{sec:tensor}, see \eqref{eq:tensor_norm}.

For $\alpha \in \R$ the polynomial weight is
\begin{equation}
    \label{eq:weight}
    w_\alpha(x):=(1+|x|^2)^{\alpha/2},
\end{equation}
and the weighted norm is
\[
    \|f\|_{\Cwa{0}}:=\sup_{x\in\R^d}\frac{|f(x)|}{w_\alpha(x)}.
\]
Following \cite{Kolokoltsov2010Nonlinear, Kolokoltsov2011Markov}, we denote by $\Cwa{0}(\R^d)$ the Banach space of continuous functions $f$ such that $\|f\|_{\Cwa{0}} < \infty$, and by $\Cwainf{0}(\R^d)$ the subspace of $f$ such that $f w_\alpha^{-1} \in C_\infty(\R^d)$, equipped with the same norm. For $\alpha \in \R$ and $k \in \N$ we use the weighted norm of order $k$
\[
    \|f\|_{\Cwa{k}}:=\max_{0\le j\le k}\sup_{x\in\R^d}\frac{\|\DD^j f(x)\|}{w_{\alpha-j}(x)}.
\]
Note that the derivative decays one power faster than the function itself. We then define the weighted  spaces
\[
    \Cwa{k}(\R^d):=\left\{f\in \Cloc{k}(\R^d): \|f\|_{\Cwa{k}}<\infty\right\},
\]
and its subspace on which $(P_t)$ will be shown strongly continuous,
\[
    \Cwainf{k}(\R^d):=\left\{
        f\in \Cloc{k}(\R^d):
        \frac{\|\DD^j f(\cdot)\|}{w_{\alpha-j}(\cdot)}\in C_\infty(\R^d),
        \quad 0\le j\le k
    \right\}.
\]
Both spaces carry the norm $\|\cdot\|_{\Cwa{k}}$. For vector- and matrix-valued maps the definitions extend componentwise, with the Euclidean norm on $\R^d$ and the Hilbert--Schmidt norm on $\R^{d\times m}$.

For the setting in which the derivatives are bounded we use a separate family of spaces. For $n \in \N_0$ set
\[
    \|f\|_{\Cwbd{n}}:=\max\Bigl\{\|f\|_{\Cwa[2]{0}},\ \max_{1\le j\le n}\sup_{x\in\R^d}\|\DD^j f(x)\|\Bigr\},
\]
with the inner maximum understood as $0$ for $n=0$, and define
\[
    \Cwbd{n}(\R^d):=\left\{f\in \Cloc{n}(\R^d) \colon \|f\|_{\Cwbd{n}}<\infty\right\}.
\]
The corresponding subspace of vanishing functions is
\[
    \Cwbdinf{n}(\R^d):=\left\{
        f\in \Cwbd{n}(\R^d):
        fw_2^{-1}\in C_\infty(\R^d),
        \quad
        \|\DD^j f(\cdot)\|\in C_\infty(\R^d),\ j=1,\dots,n
    \right\},
\]
equipped with the norm inherited from $\Cwbd{n}(\R^d)$. The condition on \(\DD^j f\) is understood componentwise, equivalently through the vanishing of its tensor norm. The superscript ``bd'' records that derivatives are bounded.

We introduce the first regime for the coefficients.

\begin{assumption}
    \label{asm:classical}
    The coefficients satisfy
    \[
        b\in \Cwa[1]{n+1}(\R^d;\R^d),
        \qquad
        \sigma\in \Cwa[1]{n+1}(\R^d;(\R^d)^m),
    \]
    and
    \[
        \|b\|_{\Cwa[1]{n+1}}\le B,
        \qquad
        \|\sigma\|_{\Cwa[1]{n+1}}\le S,
    \]
    for some $n \in \N$.
\end{assumption}

Assumption~\ref{asm:classical} at order $n$ implies Assumption~\ref{asm:bounded_derivatives} at the same $n$.

Assumption~\ref{asm:classical} holds, in particular, for geometric Brownian motion and the Ornstein--Uhlenbeck process. The hypothesis implies linear growth of $b$ and $\sigma$ and the global Lipschitz property. In particular, there exists $L > 0$ such that
\begin{equation}
    \label{eq:Lipschitz}
    |b(x) - b(y)| \leq L |x-y|, \qquad \| \sigma(x) - \sigma(y)\| \leq L |x-y|,
\end{equation}
and standard SDE theory applies; cf.\ \cite[Ch.~5]{KaratzasShreve1998}, \cite[Ch.~3]{Kunita2019}, and \cite[Theorem II.1.2]{Kunita1984SaintFlour}. In particular, \eqref{eq:SDE} has a unique strong solution, and for every $p\ge2$ there exist constants $C_p,c_p>0$ such that
\begin{equation} \label{eq:power_growth_SDE}
      \E |X_t|^p \le C_{p}(1+|x|^p) e^{c_p t}.
\end{equation}
By \eqref{eq:power_growth_SDE}, the mapping $P_t f(x):=\E_x f(X_t)$ is well-defined on $\Cwa{0}(\R^d)$.

We write $Y_t^{(r)}(x) := \DD^r X_t(x)$ for the $r$-th derivative of the SDE solution $X_t(x)$ with respect to the initial condition. These derivatives are well-defined thanks to \cite[Theorem II.3.3]{Kunita1984SaintFlour}.

We list two principal results for this regime. The first gives moment bounds for the derivatives of the SDE solution, whereas the second gives growth estimates for the derivatives of the solution to the Cauchy problem.

\begin{theorem}\label{thm:sde-A}
Fix $n \in \N$ and let Assumption~\ref{asm:classical} hold. Then for every $r \in \{1,\dots,n\}$ and $p \ge 2$, there is a constant $c > 0$ such that for all $x \in \R^d$ and all $v_1,\dots,v_r \in \R^d$,
\begin{equation}\label{eq:sde-A-bound}
    \E_x \big| w_{r-1}(X_t)\, Y^{(r)}_t(x)\,[v_1,\dots,v_r] \big|^p
    \le c\,(e^{ct}-1)\,|v_1|^p \cdots |v_r|^p, \qquad r \ge 2,
\end{equation}
and the case $r=1$ satisfies $\E_x|Y^{(1)}_t(x) v|^p \le |v|^p\,e^{ct}$.
\end{theorem}

\begin{theorem}\label{thm:cauchy-A}
Let Assumption~\ref{asm:classical} hold with $n\ge 2$, let $\alpha\in\R$, and let $f\in \Cwainf{n}(\R^d)$. Then
\[
u(t,x):=P_t f(x):=\E_x[f(X_t)]
\]
is the unique classical solution of~\eqref{eq:Cauchy} within the class of polynomial growth
\[
\sup_{0\le t\le T}\sup_{x\in\R^d}\frac{|u(t,x)|}{w_\mu(x)}<\infty,
\qquad\text{for every }T>0,\quad \mu:=\max(\alpha,1),
\]
and, for every $k\in\{0,\dots,n\}$,
\[
\|\DD^k u(t,x)\|\le \|P_t f\|_{\Cwa{n}}\,w_{\alpha-k}(x)\le e^{ct}\|f\|_{\Cwa{n}}\,w_{\alpha-k}(x).
\]
\end{theorem}

An explicit representation of $\DD^k u$ will be recorded in Lemma~\ref{lem:diff_under_expectation}.

The second assumption is dedicated to the case where the derivatives of $b$ and $\sigma$ are only required to be uniformly bounded.

\begin{assumption}
    \label{asm:bounded_derivatives}
    Fix $n\in\N$. The coefficients satisfy
    \[
        b\in \Cloc{n+1}(\R^d;\R^d),
        \qquad
        \sigma\in \Cloc{n+1}(\R^d;\R^{d\times m}),
    \]
    and
    \[
        |b(x)|\le B_0(1+|x|),
        \qquad
        \|\sigma(x)\|\le S_0(1+|x|),
        \qquad x\in\R^d,
    \]
    and
    \[
        \sup_{x\in\R^d}\|\DD^k b(x)\|\le \widetilde B,
        \qquad
        \sup_{x\in\R^d}\|\DD^k \sigma(x)\|\le \widetilde S,
        \qquad k=1,\dots,n+1.
    \]
\end{assumption}

Contrary to this condition, Assumption~\ref{asm:classical} requires decay of $\DD^j b$ and $\DD^j \sigma$ for $j \ge 2$. 

In this regime, we also give two results on derivatives of the solutions to the SDE and to the PDE, respectively.

\begin{theorem}\label{thm:sde-B}
Fix $n \in \N$ and let Assumption~\ref{asm:bounded_derivatives} hold. Then, for every $r \in \{1,\dots,n\}$ and $p \ge 2$, there is a constant $c > 0$ such that, uniformly in $x \in \R^d$,
\begin{equation}\label{eq:sde-B-bound}
    \E\,\big| Y^{(r)}_t(x)\,[v_1,\dots,v_r] \big|^p
    \le c\,(e^{ct}-1)\,|v_1|^p\cdots|v_r|^p, \qquad r \ge 2,
\end{equation}
for all $v_1,\dots,v_r \in \R^d$, and the case $r=1$ satisfies $\E|Y^{(1)}_t(x) v|^p \le |v|^p\,e^{ct}$.
\end{theorem}

\begin{theorem}\label{thm:cauchy-B}
Let Assumption~\ref{asm:bounded_derivatives} hold with $n=2$, and let $f\in \Cwbdinf{2}(\R^d)$. Consider the Cauchy problem
\[
\partial_t u(t,x)=\cL u(t,x),
\qquad
u(0,x)=f(x),
\]
where $\cL$ is defined in~\eqref{eq:generator-def}. Then
\[
u(t,x)=P_t f(x)=\E_x[f(X_t)]
\]
is the unique classical solution satisfying
\[
u\in C([0,\infty);\Cwbdinf{2}(\R^d))
\cap
C^1([0,\infty);\Cwbdinf{0}(\R^d)).
\]
Moreover, there exist $c,M\ge 0$, such that for every $t\ge 0$ and $x\in\R^d$,
\[
|u(t,x)|\le e^{ct}\|f\|_{\Cwbd{2}}\,w_2(x),
\qquad
\|\DD u(t,x)\|\le e^{ct}\|f\|_{\Cwbd{2}},
\qquad
\|\DD^2 u(t,x)\|\le \bigl(1+M t^{1/2}\bigr)e^{ct}\|f\|_{\Cwbd{2}}.
\]
\end{theorem}

The theorem is stated for $n=2$. For higher $n$, there is an obstruction shown in the following remark.

\begin{remark}\label{rem:bounded_derivatives_generator_counterexample}
We show that the generator mapping $\cL \colon \Cwbdinf{n}\to \Cwbdinf{n-2}$ fails already for $n=3$.
Already in dimension $d=1$, take $b(x)=x$ and $\sigma(x)=1$, so that Assumption~\ref{asm:bounded_derivatives} holds. Choose $\eta\in C_c^\infty(\R)$ with $\eta'(0)\ne 0$ and ${\rm supp} \, \eta \subseteq [-1,1]$. We set
\[
h(x):=\sum_{k=1}^\infty 4^{-k}\eta(x-4^k),\qquad \varphi(x):=\int_0^x h(y)\,dy .
\]
Then $\varphi$ is bounded, $\varphi'=h$, and all higher derivatives of $\varphi$ vanish at infinity. Thus $\varphi\in \Cwbdinf{n}(\R)$ for every $n$. However, for $\cL\varphi=x\varphi'+\frac12\varphi''$,
\[
(\cL\varphi)'(4^k)=\varphi'(4^k)+4^k\varphi''(4^k)+\frac12\varphi'''(4^k)\to \eta'(0)\ne 0 .
\]
Consequently $\cL\varphi\notin \Cwbdinf{1}(\R)$.
\end{remark}

Whether the precise growth rate $t^{1/2}$ in Theorem~\ref{thm:cauchy-B} is sharp is not addressed here.

\begin{example}\label{ex:moments}
Fix a multi-index $\gamma=(\gamma_1,\dots,\gamma_d)\in\N_0^d$ of order $|\gamma|:=\gamma_1+\cdots+\gamma_d$, and set
\[
    f_\gamma(x):=x^\gamma=x_1^{\gamma_1}\cdots x_d^{\gamma_d}.
\]
For every \(n\in\N_0\), one has \(\|f_\gamma\|_{\Cwa[|\gamma|]{n}}<\infty\); moreover, for every \(\varepsilon>0\), \(f_\gamma\in\Cwainf[|\gamma|+\varepsilon]{n}(\R^d)\). The mixed moment map
\[
    u_\gamma(t,x):=\E_x[X_t^\gamma]=P_tf_\gamma(x)
\]
is in $\Cloc{n}(\R^d)$, with all partial derivatives up to order $n$ of polynomial growth. Theorem~\ref{thm:cauchy-A}, applied with $\alpha=|\gamma|+\varepsilon$, gives the growth estimates of $u_\gamma$.
\end{example}

\begin{example}\label{ex:affine_polynomial}
Assume that
\[
    b(x)=b_0+Bx,\qquad
    \sigma_\ell(x)=s_\ell+S_\ell x,\qquad \ell=1,\dots,m,
\]
with $b_0,s_\ell\in\R^d$ and $B,S_\ell\in\R^{d\times d}$. Then Assumption~\ref{asm:classical} is satisfied for every $n$, and $a(x)=\sigma(x)\sigma(x)^\top$ has entries of degree at most two. If $\mathscr P_M$ denotes the polynomials on $\R^d$ of degree at most $M$, then
\[
    \cL\mathscr P_M\subseteq \mathscr P_M .
\]
Indeed, $b\cdot\nabla p$ and $\operatorname{tr}(a\DD^2p)$ have degree at most $M$. Thus, in Theorem~\ref{thm:cauchy-A} one may use weights $w_\alpha$ with $\alpha >M$.
\end{example}

\section{The Lyapunov function method} \label{sec:inv}

In this section, we provide standard Lyapunov function estimates.
\begin{lemma} \label{lemm:Lyapunov}
    Let \eqref{eq:Lipschitz} hold. Then, for any $\alpha \in \R$, there is $c_\alpha>0$ such that
    \begin{equation}
        \label{eq:Lyapunov}
        \cL w_\alpha(x) \leq c_\alpha w_\alpha(x)
    \end{equation}
\end{lemma}

The proof is a standard computation and is omitted.

\begin{lemma}\label{lemm:weighted_feller}
Assume the Lipschitz condition \eqref{eq:Lipschitz} holds. Then, for every $\alpha\in\R$, $t\ge0$, and $f\in\Cwa{0}(\R^d)$, the map $x\mapsto P_tf(x)$ is continuous. Moreover, if \(g\in C_c(\R^d)\), then \(P_tg\in C_\infty(\R^d)\) for each \(t\ge0\), and \(P_tg\to g\) uniformly on compact subsets of \(\R^d\) as \(t\downarrow0\).
\end{lemma}

\begin{proof}
Let \(x_j\to x\). The Lipschitz SDE estimate from \eqref{eq:Lipschitz} gives \(X_t(x_j)\to X_t(x)\) in probability, hence \(f(X_t(x_j))\to f(X_t(x))\) in probability. For large \(j\) the points \(x_j\) lie in a compact set. Since \(|f|\le\|f\|_{\Cwa{0}}w_\alpha\), the family \(\{f(X_t(x_j))\}\) is bounded when \(\alpha\le0\); when \(\alpha>0\), choose \(q>1\) with \(q\alpha\ge2\), and \eqref{eq:power_growth_SDE} gives \(\sup_j\E|f(X_t(x_j))|^q<\infty\). Vitali's theorem then yields \(P_tf(x_j)\to P_tf(x)\). For \(g\in C_c(\R^d)\) with \(K=\operatorname{supp}g\), the \(C_0\)-Feller property for globally Lipschitz, linearly growing SDEs gives \(\bP_x(X_t\in K)\to0\) as \(|x|\to\infty\), so \(|P_tg(x)|\le\|g\|_\infty\bP_x(X_t\in K)\to0\). For local uniform convergence as \(t\downarrow0\), fix a compact \(K\) and \(\varepsilon>0\). Uniform continuity of \(g\) gives \(\delta>0\) with \(|g(y)-g(x)|<\varepsilon\) for \(|x-y|<\delta\). The same SDE estimate gives \(\sup_{x\in K}\bP(|X_t(x)-x|\ge\delta)\to0\), hence \(\limsup_{t\downarrow0}\sup_{x\in K}|P_tg(x)-g(x)|\le\varepsilon\); let \(\varepsilon\downarrow0\).
\end{proof}

\begin{lemma} \label{lemm:invariant}
Assume \eqref{eq:Lipschitz} holds. Then, for all $\alpha \in \mathbb{R}$, the spaces $\Cwa{0}(\R^d)$
are invariant under the semigroup
\[
P_t f(x) = \mathbb{E}[f(X_t(x))]
\]
corresponding to the Cauchy problem \eqref{eq:Cauchy}, and $P_t$ is a
quasi-contraction in each of these spaces: there is $c_\alpha>0$ such that
\begin{equation}
\|P_t f\|_{\Cwa{0}} \le e^{c_\alpha t}\|f\|_{\Cwa{0}}.
\end{equation}
\end{lemma}

\begin{proof}
Applying It\^o's formula to \(w_\alpha(X_t)\) and using Lemma~\ref{lemm:Lyapunov} gives \(P_tw_\alpha(x)\le w_\alpha(x)+c_\alpha\int_0^tP_sw_\alpha(x)\,ds\); Gronwall yields \(P_tw_\alpha\le e^{c_\alpha t}w_\alpha\), and positivity of \(P_t\) applied to \(|f|\le\|f\|_{\Cwa{0}} w_\alpha\) gives the stated estimate. Lemma~\ref{lemm:weighted_feller} gives the continuity of \(P_tf\), so \(P_tf\in\Cwa{0}(\R^d)\).
\end{proof}

\begin{remark}
    The Lyapunov function approach also gives integrability of $w_\alpha(X_t)$; see \cite[Ch.~5]{Kolokoltsov2010Nonlinear}. Another route is via Picard approximations; cf.\ \cite[Theorem II.1.2]{Kunita1984SaintFlour}.
\end{remark}

\begin{proposition}
\label{theo:strongly_continuous}
The semigroup $(P_t)_{t \ge 0}$ is strongly continuous on each of the spaces
$\Cwainf{0}(\R^d)$.
\end{proposition}

\begin{proof}
Let \(f\in \Cwainf{0}(\R^d)\). Choose \(\chi\in C_c^\infty(\R^d)\) with \(\chi=1\) on \(B_1\) and \(\operatorname{supp}\chi\subset B_2\), set \(\chi_R(x)=\chi(x/R)\), \(f_R=\chi_Rf\), and \(h_R=(1-\chi_R)f\); then \(\|P_tf-f\|_{\Cwa{0}}\le\|P_th_R\|_{\Cwa{0}}+\|P_tf_R-f_R\|_{\Cwa{0}}+\|h_R\|_{\Cwa{0}}\). Since \(fw_\alpha^{-1}\in C_\infty\), \(\|h_R\|_{\Cwa{0}}\to0\), and Lemma~\ref{lemm:invariant} gives \(\sup_{0\le t\le1}\|P_th_R\|_{\Cwa{0}}\le C\|h_R\|_{\Cwa{0}}\). For the compactly supported part, fix \(\beta<\alpha\): \(\|P_tf_R\|_{\Cwa[\beta]{0}}\le C_R\) uniformly for \(0\le t\le1\), so for \(x\) outside a large compact set \(K_1\supset\mathrm{supp}\,f_R\),
\(|P_tf_R(x)-f_R(x)|/w_\alpha(x)\le C_Rw_{\beta-\alpha}(x)<\varepsilon\). On \(K_1\), Lemma~\ref{lemm:weighted_feller} gives uniform convergence \(P_tf_R\to f_R\). Hence \(\|P_tf_R-f_R\|_{\Cwa{0}}\to0\), and the decomposition concludes.
\end{proof}

\section{Random Tensor Fields and Kunita's Theory} \label{sec:tensor}

This section fixes the tensor notation for higher-order derivatives and, following Kunita's theory of the differentiable dependence of SDE solutions on the initial data, records the equations satisfied by the spatial derivatives $\DD^r X_t(x)$ of the SDE solution.

\subsection{Tensor fields}

Let $X = \R^d$ and $Y = \R^\ell$ be finite-dimensional Euclidean spaces.
For integers $r \ge 0$, denote by
\[
    \mathfrak{L}_r(X,Y)
\]
the space of continuous $r$-linear mappings
\[
    A \colon X^r \to Y.
\]
The arity index $r$ here is local to this subsection and is unrelated to the Brownian-motion dimension $m$ used in \eqref{eq:SDE}. In coordinates, an element $A \in \mathfrak{L}_r(X,\R)$ is represented by
\[
    \bigl(a_{i_1,\dots,i_r} \bigr),
    \qquad
    1 \le i_p \le d,
\]
and acts via
\[
    A[x_1,\dots,x_r]=\sum_{i_1,\dots,i_r=1}^d a_{i_1,\dots,i_r}\,x_1^{i_1} \cdots x_r^{i_r}.
\]
To avoid ambiguity, we denote the action of tensors using square brackets.

For $A\in\mathfrak{L}_r(X,Y)$, define the operator (spectral) norm
\begin{equation} \label{eq:tensor_norm}
    \|A\| :=
    \sup_{\substack{x_1,\dots,x_r\in X \\ x_i\neq 0}}
    \frac{\|A[x_1,\dots,x_r]\|}{\|x_1\|\cdots\|x_r\|}.
\end{equation}
For matrices we use the Hilbert--Schmidt (Frobenius) norm
\[
    \|A\|_{HS}^2:=\sum_{i,j}a_{ij}^2.
\]
All finite-dimensional tensor norms used below are equivalent.

For elements $x_1, \dots, x_r \in X$, we denote their tensor product by
$x_1 \otimes \dots \otimes x_r$ or $\bigotimes_{j=1}^r x_j$.
The tensor product of normed linear spaces $X_1, \dots, X_n$ is denoted by
$X_1 \otimes \dots \otimes X_n$ or $\bigotimes_{i=1}^n X_i$.
The $r$-th tensor power of $X$ is denoted by $X^{\otimes r}$.

Any tensor $A \in \mathfrak{L}_r(X,Y)$ can be viewed as an element of
$\mathcal{L}(X^{\otimes r}, Y)$.
In fact, the tensor norm on the former is equivalent to the operator norm on the latter.

For $A\in\mathfrak{L}_1(Y,Y)$ and $Z\in\mathfrak{L}_r(X,Y)$, the composition $AZ\in\mathfrak{L}_r(X,Y)$ acts on the output, $(AZ)[x_1,\dots,x_r]:=A(Z[x_1,\dots,x_r])$.

For a tensor $A \in \mathfrak{L}_r(X,Y)$, define its symmetrization as
\[
    \operatorname{Sym}(A)[x_1,\dots,x_r]:=\frac{1}{r!}\sum_{\sigma \in \mathrm{Perm}(r)}A[x_{\sigma(1)}, \dots, x_{\sigma(r)}].
\]
We define the symmetric product $x_1 \vee x_2$ as
$\operatorname{Sym}(x_1 \otimes x_2)$, and in general
\[
    x_1 \vee \dots \vee x_r:=\operatorname{Sym}\bigl(x_1 \otimes \dots \otimes x_r\bigr).
\]
When all arguments coincide, we also write
$v^{\vee r} := v \vee \dots \vee v$.

Fix $r \in \mathbb{N}$.
For a function $f \in C^r(X)$ (scalar-valued), the tensor
$\nabla f(x) \in \mathfrak{L}_1(X,\R)$ is the directional derivative
\[
    \nabla f(x)[u] :=
    \lim_{\varepsilon \to 0} \frac{f(x + \varepsilon u) - f(x)}{\varepsilon},
\]
and the tensor $\DD^r f(x) \in \mathfrak{L}_r(X,\R)$ is the iterated $r$-th order directional derivative. In coordinates,
\[
    \DD^r f(x)[u_1, \dots, u_r]=\sum_{i_1, \dots, i_r = 1}^d \frac{\partial^r f(x)}{\partial x_{i_1} \dots \partial x_{i_r}}\,u^{i_1}_1 \dots u^{i_r}_r.
\]
The tensor $\DD^r f(x)$ is symmetric, so we will sometimes write
\[
    \DD^r f(x)\bigl[ x_1 \vee \dots \vee x_r \bigr]
\]
instead of $\DD^r f(x)[x_1,\dots,x_r]$. Sometimes, for brevity, we omit the brackets, e.g. $\nabla f(x)u$.

\subsection{Kunita's theory in tensor notation} \label{sec:Kunita}

We are interested in derivatives of the solution $X_t(x)$ to \eqref{eq:SDE} with respect to the initial point $x$, that is, in the quantities $\DD^j X_t(x,\omega)$. For each fixed $t$, the solution map $x\mapsto X_t(x,\omega)$ is a random field, and its spatial derivatives can be viewed as random tensor fields; see \cite{Kunita1984SaintFlour,Kunita1990,Kunita2019}.

\begin{utheorem}[cf.\ Theorem 3.3.2 and Theorem 3.4.1 of \cite{Kunita2019}, and cf.\ Theorem II.3.1 and Lemma II.3.2 of \cite{Kunita1984SaintFlour}]
    \label{theo:deriv_exist}
    Under Assumption \ref{asm:classical}, the map $x\mapsto X_t(x)$ admits a modification continuously differentiable with respect to $x$, and the first derivative $\DD X_t(x)$ is the unique strong solution of the SDE
    \begin{equation}
        \label{eq:jet}
        \DD X_t(x)=I+\int_0^t \DD b\bigl(X_s(x)\bigr) \,\DD X_s(x)\,ds+\sum_{l=1}^m \int_0^t \DD \sigma_l\bigl(X_s(x)\bigr) \,\DD X_s(x)\,dW_s^l,
    \end{equation}
    where $I[u] = u$ is the identity mapping on $\R^d$.
    Moreover, for every $u \in \R^d$ and $p \ge 2$,
    \[
        \E \big|\DD X_t(x)[u]\big|^p < \infty.
    \]
\end{utheorem}

For derivatives of higher order we use a tensorial Fa\`a di Bruno formula.
For a finite set $A$, denote by $\mathcal{P}(A,k)$ the set of partitions of $A$
into $k$ disjoint nonempty subsets:
\[
    \mathcal{P}(A,k):=\left\{
        P = \{P_1, \dots, P_k\} :
        A = \bigcup_{i=1}^k P_i,\;
        P_i \cap P_j = \varnothing,\ i \ne j,\;
        P_i \neq \varnothing,\ i=1,\dots,k
    \right\}.
\]
For brevity, write $\mathcal{P}(r,k) := \mathcal{P}(\{1,\dots,r\},k)$.

\begin{uproposition}[Proposition 3.1 of \cite{Licht2024}]
    Let $X, Y, Z$ be Banach spaces and let $f \colon X \to Y$,
    $g \colon Y \to Z$.
    Let $x \in X$ and set $y = f(x)$.
    Assume that $f$ and $g$ have continuous derivatives up to order $m$ in a neighbourhood of $x$ and $y$, respectively. Then, for $v_1,\dots,v_m\in X$,
    \begin{equation}
    \begin{aligned}
        \DD^m (g \circ f)(x)[ v_1, \dots, v_m]
        &=
        \sum_{k=1}^{m}
        \ \sum_{P \in \mathcal{P}(m,k)}
        \DD^k g\bigl( y \bigr)
        \Bigl[
            \DD^{|P_1|} f(x)\bigl[ \bigvee_{i \in P_1} v_i \bigr],
            \dots,
            \DD^{|P_k|} f(x)\bigl[ \bigvee_{i \in P_k} v_i \bigr]
        \Bigr].
    \end{aligned}
    \end{equation}
\end{uproposition}

The preceding result can be viewed as a Fa\`a di Bruno formula in tensor form.
Recall that the classical Fa\`a di Bruno formula is a chain rule for derivatives of higher order \cite{Comtet2012}.
A multivariate version was studied in \cite{ConstantineSavits1996}; see also \cite{Savits2006,Sorokin2023} for further applications and generalizations.

We next turn to higher-order derivatives of the SDE solution $X_t(x)$. This issue was considered in the Kunita theory; cf.\ \cite[Theorem 3.4.2]{Kunita2019} and \cite[Theorem II.3.3]{Kunita1984SaintFlour}. Our contribution lies in the explicit form of the SDE for the derivatives \(\DD^r X_t(x)\) of higher order.

\begin{proposition}\label{prop:higher_jet_sde}
Fix $n\in \N$ and let Assumption~\ref{asm:bounded_derivatives} hold. For $x\in \R^d$, let $X_t(x)$ be the unique strong solution of
\[
dX_t(x)=b(X_t(x))\,dt+\sum_{\ell=1}^m \sigma_\ell(X_t(x))\,dW_t^\ell,
\qquad X_0(x)=x,
\]
where $\sigma_\ell$ denotes the $\ell$-th column of $\sigma$. For $j\ge 1$, define
\[
Y_t^{(j)}(x):=\DD_x^j X_t(x)\in \mathfrak L_j(\R^d,\R^d).
\]
Then, for every $r\in\{1,\dots,n\}$, the derivative $Y_t^{(r)}(x)$ exists almost surely, is adapted and continuous in $t$, with initial condition \(Y_0^{(1)}=I\) and \(Y_0^{(r)}=0\) for \(r\ge2\), and satisfies
\begin{equation}
\label{eq:higher-jet}
Y_t^{(r)}=Y_0^{(r)}+\int_0^t \DD_x^r\!\bigl(b\circ X_s\bigr)\,ds+\sum_{\ell=1}^m \int_0^t \DD_x^r\!\bigl(\sigma_\ell\circ X_s\bigr)\,dW_s^\ell.
\end{equation}
Moreover, for every $v_1,\dots,v_r\in \R^d$,
\begin{align}
\label{eq:Faa_di_Bruno_b}
\DD_x^r(b\circ X_s)[v_1,\dots,v_r]
&=\sum_{q=1}^r \ \sum_{\pi\in \mathcal P(r,q)}
\DD^q b(X_s)\Big[Y_s^{(|B_1|)}[v_{B_1}],\dots,Y_s^{(|B_q|)}[v_{B_q}]\Big],\\
\label{eq:Faa_di_Bruno_sigma}
\DD_x^r(\sigma_\ell\circ X_s)[v_1,\dots,v_r]
&=\sum_{q=1}^r \ \sum_{\pi\in \mathcal P(r,q)}
\DD^q \sigma_\ell(X_s)\Big[Y_s^{(|B_1|)}[v_{B_1}],\dots,Y_s^{(|B_q|)}[v_{B_q}]\Big].
\end{align}
\end{proposition}
\begin{proof}
By the results of \cite[Theorem II.3.3]{Kunita1984SaintFlour}, cf.\ \cite[Theorem 3.4.2]{Kunita2019}, there exists a modification of $x \mapsto X_t(x)$ that is $C^r$, and the derivatives
\[
Y_t^{(j)}(x)=\DD_x^j X_t(x), \qquad j=1,\dots,r,
\]
are adapted and continuous in $t$.

Fix $v_1,\dots,v_r\in \R^d$. For every $s\ge0$ and almost every $\omega$, the map
\[
x \longmapsto X_s(x,\omega)
\]
is of class $C^r$, while $b$ and $\sigma_\ell$ are of class $C^r$. We may therefore apply Proposition 3.1 of \cite{Licht2024} pathwise to the compositions $b\circ X_s(\cdot,\omega)$ and $\sigma_\ell\circ X_s(\cdot,\omega)$. This yields \eqref{eq:Faa_di_Bruno_b} and \eqref{eq:Faa_di_Bruno_sigma}.

Approximating the derivatives by finite differences as in \cite[\S3.4]{Kunita2019} gives \eqref{eq:higher-jet}.
\end{proof}

We record several equivalent formulations of the above equations. Introduce the notation
\[
\cY_t:=\bigl(X_t,Y_t^{(1)},\dots,Y_t^{(n)}\bigr).
\]
For each $r\in\{1,\dots,n\}$ and $\ell\in\{1,\dots,m\}$, define the $r$-linear coefficients
\[
b_r(\cY_t):=\DD_x^r(b\circ X_t),
\qquad
\sigma_{r,\ell}(\cY_t):=\DD_x^r(\sigma_\ell\circ X_t).
\]
Thus, for every $v_1,\dots,v_r\in\R^d$,
\begin{align*}
b_r(\cY_t)[v_1,\dots,v_r]
&=\sum_{q=1}^r \ \sum_{\pi\in \mathcal P(r,q)}
\DD^q b(X_t)\Big[Y_t^{(|B_1|)}[v_{B_1}],\dots,Y_t^{(|B_q|)}[v_{B_q}]\Big],\\
\sigma_{r,\ell}(\cY_t)[v_1,\dots,v_r]
&=\sum_{q=1}^r \ \sum_{\pi\in \mathcal P(r,q)}
\DD^q \sigma_\ell(X_t)\Big[Y_t^{(|B_1|)}[v_{B_1}],\dots,Y_t^{(|B_q|)}[v_{B_q}]\Big].
\end{align*}
In this notation, \eqref{eq:higher-jet} reads
\[
Y_t^{(r)}=Y_0^{(r)}+\int_0^t b_r(\cY_s)\,ds+\sum_{\ell=1}^m \int_0^t \sigma_{r,\ell}(\cY_s)\,dW_s^\ell.
\]
Equivalently, $Y_t^{(r)}$ satisfies the SDE
\begin{equation}
\label{eq:Yr_linear_rhs}
Y_t^{(r)}=Y_0^{(r)}+\int_0^t \Big( \DD b(X_s)\,Y_s^{(r)} + F_s^{(r)} \Big)\,ds+\sum_{\ell=1}^m \int_0^t \Big( \DD \sigma_\ell(X_s)\,Y_s^{(r)} + G_{\ell,s}^{(r)} \Big)\,dW_s^\ell,
\end{equation}
where \(F_s^{(1)}=0\), \(G_{\ell,s}^{(1)}=0\), and, for $r\ge2$,
\begin{align}
\label{eq:Fr_def}
F_s^{(r)}[v_1,\dots,v_r]
&:=\sum_{q=2}^r \ \sum_{\pi\in \mathcal P(r,q)}
\DD^q b(X_s)\Big[Y_s^{(|B_1|)}[v_{B_1}],\dots,Y_s^{(|B_q|)}[v_{B_q}]\Big],\\
\label{eq:Gr_def}
G_{\ell,s}^{(r)}[v_1,\dots,v_r]
&:= \sum_{q=2}^r \ \sum_{\pi\in \mathcal P(r,q)}
\DD^q \sigma_\ell(X_s)\Big[Y_s^{(|B_1|)}[v_{B_1}],\dots,Y_s^{(|B_q|)}[v_{B_q}]\Big].
\end{align}
Here $\pi=\{B_1,\dots,B_q\}$ runs over the partitions of $\{1,\dots,r\}$ into $q$ nonempty blocks, and $v_{B_j}$ denotes the collection of vectors $(v_i)_{i\in B_j}$.

\section{Estimates for spatial derivatives of diffusions} \label{sec:derivative}

This section establishes the moment bounds for the spatial derivatives $\DD^r X_t(x)$ of the SDE solution, under each of Assumptions~\ref{asm:classical} and~\ref{asm:bounded_derivatives}.

\begin{lemma}\label{lemm:ito_gronwall}
Let \(E\) be a finite-dimensional Euclidean space (in our applications \(E=\R^d\) or \(E=\mathfrak L_r(\R^d,\R^d)\) with the Hilbert--Schmidt inner product, or a tensor product of these). Let \(p\ge2\), and let \(M_t\in E\) be an It\^o process satisfying
\[
    dM_t=\bigl(\mathfrak A_tM_t+H_t\bigr)\,dt+\sum_{\ell=1}^m\bigl(\mathfrak B_{\ell,t}M_t+K_{\ell,t}\bigr)\,dW_t^\ell,\qquad M_0\in E,
\]
with adapted coefficients \(\mathfrak A_t,\mathfrak B_{\ell,t}:E\to E\) satisfying \(\|\mathfrak A_t\|+\sum_\ell\|\mathfrak B_{\ell,t}\|\le L\), and adapted processes \(H_t,K_{\ell,t}\in E\) with
\[
    \E|H_s|^p+\sum_{\ell=1}^m\E|K_{\ell,s}|^p\le\varphi(s),\qquad 0\le s\le T,
\]
for some non-decreasing, locally integrable \(\varphi:[0,T]\to[0,\infty)\). Then there is a constant \(C>0\) such that
\begin{equation}\label{eq:ito_gronwall}
    \E|M_t|^p\le\Bigl(\E|M_0|^p+\int_0^t\varphi(s)\,ds\Bigr)e^{Ct},\qquad 0\le t\le T.
\end{equation}
\end{lemma}

\begin{proof}
Fix \(\varepsilon\in(0,1]\) and set \(V_t^\varepsilon:=(\varepsilon+|M_t|^2)^{p/2}\), \(\tau_N^\varepsilon:=\inf\{t:V_t^\varepsilon\ge N\}\). It\^o's formula applied to the smooth function \(z\mapsto(\varepsilon+|z|^2)^{p/2}\), combined with \(\|\mathfrak A_t\|+\sum_\ell\|\mathfrak B_{\ell,t}\|\le L\) and Young's inequality, yields
\[
    \E V_{t\wedge\tau_N^\varepsilon}^\varepsilon\le\E V_0^\varepsilon+C\int_0^t\bigl(\E V_{s\wedge\tau_N^\varepsilon}^\varepsilon+\varphi(s)\bigr)\,ds,
\]
where \(C>0\). Gronwall's inequality gives
\(\E V_{t\wedge\tau_N^\varepsilon}^\varepsilon\le\bigl(\E V_0^\varepsilon+\int_0^t\varphi(s)\,ds\bigr)e^{Ct}\).
Letting \(N\to\infty\) (Fatou) and then \(\varepsilon\downarrow0\) yields \eqref{eq:ito_gronwall}.
\end{proof}

\begin{proposition} \label{lemm:multi_d_1st}
    Let \(b\) and the vector fields \(\sigma_1,\dots,\sigma_m\) in \eqref{eq:SDE} have bounded first derivatives. Set
    \[
        B_1 := \sup_{x} \| \DD b(x) \|, \qquad
        S_1 := \left(\sum_{\ell=1}^m \sup_x\|\DD\sigma_\ell(x)\|^2\right)^{1/2}.
    \]
    Let $Y_t = \DD X_t$.
    Then for $p > 0$
    \begin{equation} \label{eq:multi_1st}
        \E |Y_t v|^p \leq |v|^p e^{c t},
    \end{equation}
    where
    \begin{eqnarray*}
        c =
        \begin{cases}
            p B_1 + \frac{p(p-1)}{2} S_1^2, \qquad p \ge 2, \\
            \frac{p}{2} \left(2 B_1 +  S_1^2\right), \qquad 0 < p < 2.
        \end{cases}
    \end{eqnarray*}
\end{proposition}
\begin{proof}
Fix \(v\in\R^d\) and set \(M_t:=Y_tv\). By Proposition~\ref{prop:higher_jet_sde}, we have the SDE
\begin{equation}\label{eq:Mtv_SDE}
    M_t=v+\int_0^t \DD b(X_s)M_s\,ds+\sum_{\ell=1}^m\int_0^t\DD\sigma_\ell(X_s)M_s\,dW_s^\ell .
\end{equation}

\emph{Case \(p\ge2\).} Fix \(\varepsilon\in(0,1]\) and let \(f_\varepsilon(z):=(\varepsilon+|z|^2)^{p/2}\). Then \(f_\varepsilon\in C^2(\R^d)\) with
\[
    \nabla f_\varepsilon(z)=p(\varepsilon+|z|^2)^{p/2-1}\,z,\qquad
    \DD^2 f_\varepsilon(z)=p(\varepsilon+|z|^2)^{p/2-1}I_d+p(p-2)(\varepsilon+|z|^2)^{p/2-2}\,zz^\top .
\]
By It\^o's formula applied to \eqref{eq:Mtv_SDE}, the drift of \(f_\varepsilon(M_t)\) is the sum of the first-order term
\[
    \nabla f_\varepsilon(M_s)\cdot\DD b(X_s)M_s
    \le p(\varepsilon+|M_s|^2)^{p/2-1}\|\DD b(X_s)\||M_s|^2
    \le pB_1f_\varepsilon(M_s)
\]
and the second-order term
\[
    \tfrac12\sum_{\ell=1}^m\DD^2 f_\varepsilon(M_s)\bigl[\DD\sigma_\ell(X_s)M_s,\DD\sigma_\ell(X_s)M_s\bigr]
    \le \tfrac{p(p-1)}{2}S_1^2 f_\varepsilon(M_s);
\]
hence this drift is at most \(c\,f_\varepsilon(M_s)\) with \(c:=pB_1+\tfrac{p(p-1)}{2}S_1^2\). Set \(\tau_N:=\inf\{t\ge0:f_\varepsilon(M_t)\ge N\}\); on \([0,\tau_N]\) the It\^o stochastic integral has bounded integrand, hence is a true martingale, and taking expectations gives
\[
    \E f_\varepsilon(M_{t\wedge\tau_N})\le f_\varepsilon(v)+c\int_0^t\E f_\varepsilon(M_{s\wedge\tau_N})\,ds.
\]
Gronwall's lemma gives \(\E f_\varepsilon(M_{t\wedge\tau_N})\le(\varepsilon+|v|^2)^{p/2}e^{ct}\). Fatou's lemma and monotone convergence yield \(\E|M_t|^p\le|v|^p e^{ct}\).

\emph{Case \(0<p<2\).} The map \(x\mapsto x^{p/2}\) is concave on \([0,\infty)\), so Jensen's inequality gives
\[
    \E|M_t|^p=\E\bigl[(|M_t|^2)^{p/2}\bigr]\le(\E|M_t|^2)^{p/2}\le|v|^p e^{c_2pt/2}.
\]
Hence \eqref{eq:multi_1st} holds with \(c=\tfrac{p}{2}(2B_1+S_1^2)\).
\end{proof}

\begin{lemma}\label{prop:multid_bounded_derivatives}
Let Assumption \ref{asm:bounded_derivatives} hold with $n\ge2$. Let $r\in\{2,\dots,n\}$ and $p>0$. Then there exists a constant $c_{r,p}>0$ such that
\[
\sup_{x\in\R^d}\sup_{|v_1|,\dots,|v_r|\le1}
\E_x \big|Y_t^{(r)}(x)[v_1,\dots,v_r]\big|^p
\le c_{r,p}\bigl(e^{c_{r,p} t}-1\bigr).
\]
\end{lemma}

\begin{proof}
For \(j\ge1\), write
\[
\widetilde M_{j,p}(t):=\sup_{x\in\R^d}\sup_{|v_1|,\dots,|v_j|\le1}\E_x\big|Y_t^{(j)}(x)[v_1,\dots,v_j]\big|^p .
\]
The case \(j=1\) is Proposition~\ref{lemm:multi_d_1st}, which gives \(\widetilde M_{1,p}(t)\le e^{ct}\). The case \(0<p<2\) follows from \(p\ge2\) by the H\"older inequality, so it suffices to take \(p\ge2\). Fix \(r\in\{2,\dots,n\}\) and assume the bound for orders \(<r\). Set \(U_t^{(r)}:=Y_t^{(r)}(x)[v_1,\dots,v_r]\) with \(|v_i|\le1\). Under Assumption~\ref{asm:bounded_derivatives}, \(\|\DD^qb\|+\max_\ell\|\DD^q\sigma_\ell\|\le c\) uniformly for \(q=1,\dots,r\), so \eqref{eq:Yr_linear_rhs} puts \(U_t^{(r)}\) in the form of Lemma~\ref{lemm:ito_gronwall} with bounded linear coefficients and additional terms \(F_s^{(r)}, G_{\ell,s}^{(r)}\). For a partition with \(q\ge2\) blocks, the H\"older inequality and the hypothesis of induction give
\[
    \E_x\big|\DD^qb(X_s)[Y_s^{(|B_1|)}[v_{B_1}],\dots,Y_s^{(|B_q|)}[v_{B_q}]]\big|^p\le c\,e^{cs}.
\]
Summing over the partitions in \eqref{eq:Fr_def} and \eqref{eq:Gr_def} gives
\[
    \E_x|F_s^{(r)}[v]|^p+\sum_\ell\E_x|G_{\ell,s}^{(r)}[v]|^p\le c\,e^{cs} .
\]
Lemma~\ref{lemm:ito_gronwall} with \(U_0^{(r)}=0\) and \(\varphi(s)=c\,e^{cs}\) yields \(\widetilde M_{r,p}(t)\le c(e^{ct}-1)\).
\end{proof}

\begin{lemma} \label{lemm:normalized_Z_sde}
Let Assumption \ref{asm:classical} hold, \(r\in\{1,\dots,n\}\), and let \(Z_t^{(r)}:=w_{r-1}(X_t)Y_t^{(r)}\) be the normalized $r$-th derivative of the SDE solution. Then
\begin{equation}\label{eq:normalized_Z_sde}
    dZ_t^{(r)}=\bigl(\mathfrak A_t^{(r)}Z_t^{(r)}+\widehat F_t^{(r)}\bigr)\,dt+\sum_{\ell=1}^m\bigl(\mathfrak B_{\ell,t}^{(r)}Z_t^{(r)}+\widehat G_{\ell,t}^{(r)}\bigr)\,dW_t^\ell,
\end{equation}
with \(Z_0^{(1)}=I\), \(Z_0^{(r)}=0\) for \(r\ge2\), and uniformly bounded linear coefficients
\[
\begin{aligned}
    \mathfrak A_t^{(r)}&=\DD b(X_t)+\frac{\cL w_{r-1}(X_t)}{w_{r-1}(X_t)}\,I+\sum_{\ell=1}^m\frac{\nabla w_{r-1}(X_t)\sigma_\ell(X_t)}{w_{r-1}(X_t)}\,\DD\sigma_\ell(X_t),\\
    \mathfrak B_{\ell,t}^{(r)}&=\DD\sigma_\ell(X_t)+\frac{\nabla w_{r-1}(X_t)\sigma_\ell(X_t)}{w_{r-1}(X_t)}\,I.
\end{aligned}
\]
For \(r\ge2\),
\[
    \widehat G_{\ell,t}^{(r)}:=w_{r-1}(X_t)G_{\ell,t}^{(r)},\qquad
    \widehat F_t^{(r)}:=w_{r-1}(X_t)F_t^{(r)}+\sum_{\ell=1}^m\bigl(\nabla w_{r-1}(X_t)\sigma_\ell(X_t)\bigr)G_{\ell,t}^{(r)};
\]
both are zero at \(r=1\). For all \(v_1,\dots,v_r\in\R^d\),
\begin{equation}\label{eq:normalized_Z_rhs_bound}
    |\widehat F_t^{(r)}[v_1,\dots,v_r]|+\sum_{\ell=1}^m|\widehat G_{\ell,t}^{(r)}[v_1,\dots,v_r]|
    \le C\sum_{q=2}^r\sum_{\pi=\{B_1,\dots,B_q\}\in\mathcal P(r,q)}
    \prod_{j=1}^q |Z_t^{(|B_j|)}[v_{B_j}]|.
\end{equation}
\end{lemma}

\begin{proof}
The equation \eqref{eq:normalized_Z_sde} follows by applying It\^o's product rule to \(w_{r-1}(X_t)Y_t^{(r)}\) and using \eqref{eq:Yr_linear_rhs}. Since
\[
    dw_{r-1}(X_t)=\cL w_{r-1}(X_t)\,dt+\sum_{\ell=1}^m\nabla w_{r-1}(X_t)\sigma_\ell(X_t)\,dW_t^\ell,
\]
the quadratic covariation between \(w_{r-1}(X_t)\) and \(Y_t^{(r)}\) contributes
\[
    \sum_{\ell=1}^m\bigl(\nabla w_{r-1}(X_t)\sigma_\ell(X_t)\bigr)\bigl(\DD\sigma_\ell(X_t)Y_t^{(r)}+G_{\ell,t}^{(r)}\bigr)\,dt.
\]
Collecting the parts linear in \(Y_t^{(r)}=w_{1-r}(X_t)Z_t^{(r)}\) gives the displayed coefficients \(\mathfrak A_t^{(r)}\) and \(\mathfrak B_{\ell,t}^{(r)}\), while what remains is precisely \(\widehat F_t^{(r)}\) and \(\widehat G_{\ell,t}^{(r)}\).

The boundedness of \(\mathfrak A_t^{(r)}\) and \(\mathfrak B_{\ell,t}^{(r)}\) follows from
\[
    \frac{|\cL w_{r-1}(y)|}{w_{r-1}(y)}+\sum_{\ell=1}^m\frac{|\nabla w_{r-1}(y)\sigma_\ell(y)|}{w_{r-1}(y)}+\|\DD b(y)\|+\sum_{\ell=1}^m\|\DD\sigma_\ell(y)\|\le C,
\]
which is a direct consequence of Assumption \ref{asm:classical} and the elementary bounds \(|\nabla w_{r-1}|\le Cw_{r-1}w_{-1}\), \(\|\DD^2w_{r-1}\|\le Cw_{r-1}w_{-2}\).

For \eqref{eq:normalized_Z_rhs_bound}, consider a summand of \(w_{r-1}(X_t)F_t^{(r)}\) for a partition \(\pi=\{B_1,\dots,B_q\}\) with \(q\ge2\), namely
\[
    w_{r-1}(X_t)\,\DD^q b(X_t)\bigl[Y_t^{(|B_1|)}[v_{B_1}],\dots,Y_t^{(|B_q|)}[v_{B_q}]\bigr].
\]
Write each \(Y_t^{(|B_j|)}=w_{1-|B_j|}(X_t)Z_t^{(|B_j|)}\) and use \(\|\DD^q b(X_t)\|\le Cw_{1-q}(X_t)\) (Assumption~\ref{asm:classical}). The summand has the weight 
\[
    w_{r-1}(X_t)\,w_{1-q}(X_t)\prod_{j=1}^qw_{1-|B_j|}(X_t)=w_0(X_t)=1,
\]
the exponents summing to \((r-1)+(1-q)+\sum_{j=1}^q(1-|B_j|)=0\) since \(\sum_j|B_j|=r\). Here \(w_{r-1}\) is the normalization, \(w_{1-q}\) is the bound on \(\DD^q b\), and each \(w_{1-|B_j|}\) normalizes the jet \(Y_t^{(|B_j|)}\); their exact cancellation to \(w_0=1\) is what leaves \(\widehat F_t^{(r)}\) a bounded combination of the normalized derivatives \(Z_t^{(|B_j|)}\). The summands of \(w_{r-1}G_\ell^{(r)}\) are treated identically, with \(\sigma_\ell\) in place of \(b\).

The remaining term in the formula for \(\widehat F_t^{(r)}\), namely \((\nabla w_{r-1}(X_t)\sigma_\ell(X_t))\,G_{\ell,t}^{(r)}\), is treated in a similar way. Summing over the partitions gives \eqref{eq:normalized_Z_rhs_bound}.
\end{proof}

\begin{lemma} \label{lemm:multid}
    Let $b$ and $\sigma$ satisfy Assumption \ref{asm:classical} for $n \geq 2$.
    Then, for every \(r\in\{2,\dots,n\}\), \(p > 0\), and \(v_1,\dots,v_r\in\R^d\),
    \begin{equation} \label{eq:multi_higher}
        \E_x |Y_t^{(r)}(x)[v_1,\dots,v_r]|^p \leq \frac{c(e^{c t} - 1)|v_1|^p\cdots |v_r|^p}{(1+ |x|^2)^{p(r-1)/2}},
    \end{equation}
    where $c>0$ is a constant.
    Let \(Z_t^{(r)}=w_{r-1}(X_t)Y_t^{(r)}\). Then, for all \(p\ge2\),
    \begin{equation} \label{eq:higher_Z}
        \E_x |Z_t^{(r)}(x)[v_1,\dots,v_r]|^p \leq c (e^{c t} - 1)|v_1|^p\cdots |v_r|^p.
    \end{equation}
\end{lemma}

\begin{proof}
By induction on \(r\). The case \(r=1\) is Proposition~\ref{lemm:multi_d_1st}. Take \(p\ge2\) (\(0<p<2\) follows by the H\"older inequality), fix \(r\ge2\) and \(|v_i|\le1\).

Set \(U_t:=Y_t^{(r)}[v_1,\dots,v_r]\). The SDE \eqref{eq:Yr_linear_rhs} fits Lemma~\ref{lemm:ito_gronwall} with linear coefficients \(\DD b(X_t),\DD\sigma_\ell(X_t)\), bounded by Assumption~\ref{asm:classical}, and additional terms \(F_s^{(r)},G_{\ell,s}^{(r)}\). Fix a partition \(\pi=\{B_1,\dots,B_q\}\) with \(q\ge2\); Assumption~\ref{asm:classical} gives \(\|\DD^qb(X_s)\|\le c\,w_{1-q}(X_s)\). Apply the H\"older inequality at exponent \(p(q+1)\) to
\[
    R_{\pi,s}:=\DD^qb(X_s)[Y_s^{(|B_1|)}[v_{B_1}],\dots,Y_s^{(|B_q|)}[v_{B_q}]].
\]
Using Lemma~\ref{lemm:invariant} (\(\E_xw_{1-q}(X_s)^{p(q+1)}\le c\,e^{cs}w_{1-q}(x)^{p(q+1)}\)) and the hypothesis of induction, we obtain
\[
    \E_x|R_{\pi,s}|^p\le c\,e^{cs}\,w_{1-q}(x)^p\prod_j w_{1-|B_j|}(x)^p
    =c\,e^{cs}\,w_{1-r}(x)^p,
\]
since \(\sum_j|B_j|=r\). Summing in \eqref{eq:Fr_def}, \eqref{eq:Gr_def} and applying Lemma~\ref{lemm:ito_gronwall} with \(U_0=0\), \(\varphi(s)=c\,e^{cs}w_{1-r}(x)^p\) gives \eqref{eq:multi_higher}.

The inequality \eqref{eq:higher_Z} is proven similarly, thanks to Lemma~\ref{lemm:normalized_Z_sde} and Lemma~\ref{lemm:ito_gronwall}.
\end{proof}

We record two basic bounds on the normalized derivatives \(J_q f\). Only one derivative of \(J_qf\) enters, so both bounds apply in the borderline case \(q=n-1\).

\begin{lemma} \label{lem:J_increment}
Let \(n\in\N\), \(\alpha\in\R\), and \(f\in\Cwa{n}(\R^d)\). For \(q=0,\dots,n\), write
\[
    J_qf(x):=w_{q-\alpha}(x)\DD^q f(x).
\]
Then \(J_qf\) is bounded for every \(q=0,\dots,n\). Moreover, if \(q\le n-1\), then
\[
    \|\DD J_qf(y)\|\le C\|f\|_{\Cwa{n}}w_{-1}(y),\qquad y\in\R^d,
\]
for some \(C>0\).
\end{lemma}

\begin{proof}
The boundedness of \(J_qf\) follows directly from the definition of \(\|f\|_{\Cwa{n}}\). If \(q\le n-1\), then the Leibniz rule gives
\[
    \DD J_qf=(\DD w_{q-\alpha})\otimes\DD^qf+w_{q-\alpha}\DD^{q+1}f.
\]
Using \(|\DD w_{q-\alpha}|\le Cw_{q-\alpha}w_{-1}\), \(\|\DD^qf\|\le \|f\|_{\Cwa{n}}w_{\alpha-q}\), and \(\|\DD^{q+1}f\|\le \|f\|_{\Cwa{n}}w_{\alpha-q-1}\), we obtain the required inequality.
\end{proof}

\begin{lemma}\label{lem:J_increment_stochastic}
Let \(n\in\N\), \(\alpha\in\R\), and let Assumption \ref{asm:classical} hold. Let \(f\in\Cwa{n}(\R^d)\), and set
\[
    J_qf(x):=w_{q-\alpha}(x)\DD^qf(x),\qquad q=0,\dots,n .
\]
Then, for every \(q\le n-1\) and every \(p\ge2\), there is a constant \(C>0\) such that
\begin{equation}\label{eq:J_increment_stochastic}
    \E_x\|J_qf(X_t)-J_qf(x)\|^p
    \le C(e^{Ct}-1)\|f\|_{\Cwa{n}}^p .
\end{equation}
\end{lemma}

\begin{proof}
The proof has two steps: a moment bound on the normalized displacement \((X_t-x)/w_1(x)\), and a pointwise bound on the increment of \(J_qf\). Combining the two gives the lemma.

\emph{Step 1: displacement moment bound.}
Set \(\Theta_t:=(X_t-x)/w_1(x)\). Since \(X_t=x+w_1(x)\Theta_t\), Assumption~\ref{asm:classical} gives \(\bigl(|b(X_t)|+\|\sigma(X_t)\|\bigr)/w_1(x)\le C(1+|\Theta_t|)\). Apply It\^o's formula to \(V_t^\varepsilon:=(\varepsilon+|\Theta_t|^2)^{p/2}\) for \(\varepsilon\in(0,1]\) and stop at \(\tau_N^\varepsilon:=\inf\{t:V_t^\varepsilon\ge N\}\). Gronwall gives
\[
    \E_xV_{t\wedge\tau_N^\varepsilon}^\varepsilon\le \varepsilon^{p/2}e^{Ct}+C(e^{Ct}-1).
\]
Letting \(N\to\infty\) (Fatou) and \(\varepsilon\downarrow0\) yields
\begin{equation}\label{eq:normalized_displacement}
    \E_x\Bigl(\tfrac{|X_t-x|}{w_1(x)}\Bigr)^p\le C(e^{Ct}-1).
\end{equation}

\emph{Step 2: pointwise increment bound and combination.}
By Lemma~\ref{lem:J_increment}, \(\|\DD J_qf(y)\|\le C\|f\|_{\Cwa{n}}w_{-1}(y)\) and \(\|J_qf\|_\infty\le\|f\|_{\Cwa{n}}\). The gradient bound, used when \(|y-x|\le w_1(x)/2\), together with the sup bound, used otherwise, gives
\[
    \|J_qf(y)-J_qf(x)\|\le C\|f\|_{\Cwa{n}}\Bigl(\tfrac{|y-x|}{w_1(x)}\wedge1\Bigr).
\]
Applying this with \(y=X_t\), using \((a\wedge1)^p\le a^p\), and invoking \eqref{eq:normalized_displacement}, we obtain \eqref{eq:J_increment_stochastic}.
\end{proof}

The following lemma gives mixed moment estimates for products of normalized derivatives, optionally twisted by a transported polynomial weight.

\begin{lemma}\label{lem:weighted_product_multi}
Let Assumption~\ref{asm:classical} hold, \(r_1,\dots,r_k\in\{1,\dots,n\}\) with \(\sum_ir_i\le n\), \(\beta\in\R\), \(p\ge2\), and
\[
    R_t^\beta:=\frac{w_\beta(X_t)}{w_\beta(x)},\qquad
    U_t:=\bigotimes_{i=1}^k Z_t^{(r_i)}[v_i],
    \qquad v_i\in(\R^d)^{r_i},\ |v_{i,a}|\le1.
\]
There is \(C>0\) such that, for all \(t\ge0\),
\begin{equation}\label{eq:weighted_product_multi}
    \E_x\left|R_t^\beta U_t\right|^p
    \le
    \begin{cases}
    Ce^{Ct}, & \max_ir_i=1,\\[1.1ex]
    C(e^{Ct}-1), & \max_ir_i\ge2.
    \end{cases}
\end{equation}
\end{lemma}

\begin{proof}
Set \(h:=\#\{i:r_i=\max_ir_i\}\); splitting the indices into \(h\) groups with one factor of maximal order each, factoring \(R_t^\beta=\prod_{a=1}^hR_t^{\beta/h}\), and replacing \(p\) by \(ph\) reduces to \(h=1\).

\emph{Case \(\max_ir_i=1\).} Here \(Z^{(1)}=Y^{(1)}\); the H\"older inequality on the \(k+1\) factors \(R_t^\beta,Z_t^{(1)}[v_1],\dots,Z_t^{(1)}[v_k]\), Lemma~\ref{lemm:invariant}, and Proposition~\ref{lemm:multi_d_1st} give \(Ce^{Ct}\).

\emph{Case \(\max_ir_i=r_1\ge2\)} (assuming the claim for smaller maximal order). By It\^o's formula, \(R_t^\beta\) satisfies a linear SDE with coefficients bounded by Assumption~\ref{asm:classical}. The It\^o product rule and Lemma~\ref{lemm:normalized_Z_sde} then cast \(R_t^\beta U_t\) into the form of Lemma~\ref{lemm:ito_gronwall}. The hypothesis of induction (or Case~1 if every factor has order one) bounds \(\E_x|R^\beta\widetilde U|^p\) by \(Ce^{Cs}\). Since \(R_0^\beta U_0=0\), Lemma~\ref{lemm:ito_gronwall} gives the bound.
\end{proof}

\begin{lemma}\label{lem:weighted_product_expectation}
Let \(n\in\N\), let Assumption~\ref{asm:classical} hold, let \(a\in\{2,\dots,n\}\), let \(q\in\N\) satisfy \(a+q-1\le n\), and let \(\beta\in\R\). For unit vectors \(v_1,\dots,v_a,u_1,\dots,u_{q-1}\in\R^d\), set
\[
    V_t:=R_t^\beta Z_t^{(a)}[v_1,\dots,v_a]\otimes\bigotimes_{j=1}^{q-1}Z_t^{(1)}[u_j],
\]
with the empty tensor product interpreted as the scalar \(1\). Then
\begin{equation}\label{eq:weighted_product_expectation}
    \left|\E_x V_t\right|\le C(e^{Ct}-1),
    \qquad t\ge0,
\end{equation}
for some \(C>0\).
\end{lemma}

\begin{proof}
Since \(Z_0^{(a)}=0\) for \(a\ge2\) by Lemma~\ref{lemm:normalized_Z_sde}, \(V_0=0\). The It\^o product rule, applied as in the proof of Lemma~\ref{lem:weighted_product_multi}, gives an SDE
\[
    dV_t=\bigl(\mathfrak A_t V_t+\widetilde F_t\bigr)\,dt+\sum_\ell\bigl(\mathfrak B_{\ell,t} V_t+\widetilde G_{\ell,t}\bigr)\,dW_t^\ell,
\]
with uniformly bounded linear coefficients \(\mathfrak A_t,\mathfrak B_{\ell,t}\). The terms \(\widetilde F_t,\widetilde G_{\ell,t}\) come from the additional term in the SDE for \(Z^{(a)}\) (Lemma~\ref{lemm:normalized_Z_sde}); each has the form of \(V_t\) with the factor \(Z^{(a)}\) replaced by a product \(Z^{(s_1)}\otimes\cdots\otimes Z^{(s_h)}\) with \(h\ge2\), \(s_i<a\), and \(\sum_i s_i=a\). Taking expectations eliminates the stochastic integrals, so
\[
    |\E_x V_t|
    \le \int_0^t\|\mathfrak A_s\|\,\E_x|V_s|\,ds+\int_0^t \E_x|\widetilde F_s|\,ds.
\]
Lemma~\ref{lem:weighted_product_multi} applied to \(V_s\) (max-order \(a\ge2\)) gives \(\E_x|V_s|\le C(e^{Cs}-1)^{1/2}\); applied to each summand of \(\widetilde F_s\), whose index of maximal order is \(<a\), it gives \(\E_x|\widetilde F_s|\le Ce^{Cs}\). Hence
\[
    |\E_x V_t|\le C\int_0^t (e^{Cs}-1)^{1/2}\,ds+C\int_0^t e^{Cs}\,ds
    \le C(e^{Ct}-1),
\]
where the last step uses \(\int_0^t(e^{Cs}-1)^{1/2}\,ds\le C(e^{Ct}-1)\), bounding the integrand by \((Cs)^{1/2}\) for small \(s\) and by \(e^{Cs/2}\) for large \(s\).
\end{proof}

\begin{proposition}\label{prop:weighted_non_singleton_cancellation}
Let Assumption \ref{asm:classical} hold. Fix \(r\in\{2,\dots,n\}\), let \(\pi=\{B_1,\dots,B_q\}\in\mathcal P(r,q)\) be a non-singleton partition, and write \(s_j=|B_j|\). Then, for every \(f\in\Cwa{n}(\R^d)\),
\[
\sup_{|v_1|,\dots,|v_r|\le1}
\left|
\E_x\left[
\frac{w_{\alpha-r}(X_t)}{w_{\alpha-r}(x)}
J_qf(X_t)
\big[
Z_t^{(s_1)}[v_{B_1}],\dots,Z_t^{(s_q)}[v_{B_q}]
\big]\right]\right|
\le C(e^{Ct}-1)\|f\|_{\Cwa{n}}.
\]
\end{proposition}

\begin{proof}
Put \(\beta=\alpha-r\) and \(R_t^\beta=w_\beta(X_t)/w_\beta(x)\). If at least two blocks of \(\pi\) have size at least two, relabel them as \(B_1\) and \(B_2\), set \(U_j:=Z_t^{(s_j)}[v_{B_j}]\), and split the remaining indices \(\{3,\dots,q\}\) into two disjoint sets \(I\) and \(I^c\). Since \(R_t^\beta=R_t^{\beta/2}R_t^{\beta/2}\), Cauchy's inequality gives
\[
\begin{aligned}
&\left|
\E_x\left[
R_t^\beta J_qf(X_t)[U_1,\dots,U_q]
\right]\right|  \\
&\qquad\le
\|f\|_{\Cwa{n}}
\left(\E_x\left|R_t^{\beta/2}U_1\otimes\bigotimes_{j\in I}U_j\right|^2\right)^{1/2}
\left(\E_x\left|R_t^{\beta/2}U_2\otimes\bigotimes_{j\in I^c}U_j\right|^2\right)^{1/2}.
\end{aligned}
\]
Each \(L^2\)-norm contains one normalized derivative of higher order, hence is bounded by \(C(e^{Ct}-1)^{1/2}\) by Lemma~\ref{lem:weighted_product_multi}, yielding the stated bound.

If exactly one block has size \(a\ge2\) and the rest are singletons, relabel so that \(s_1=a\) and write
\[
    J_qf(X_t)=J_qf(x)+\bigl(J_qf(X_t)-J_qf(x)\bigr).
\]
The first term is estimated by Lemma \ref{lem:weighted_product_expectation}. For the second term, Cauchy's inequality, Lemma \ref{lem:J_increment_stochastic}, and Lemma \ref{lem:weighted_product_multi} yield
\[
\begin{aligned}
&\E_x\left[
R_t^\beta\|J_qf(X_t)-J_qf(x)\|
|Z_t^{(a)}[v_{B_1}]|
\prod_{j=2}^q|Z_t^{(1)}[v_{B_j}]|
\right]  \\
&\qquad\le
\left(\E_x\|J_qf(X_t)-J_qf(x)\|^2\right)^{1/2}
\left(\E_x\left|R_t^\beta Z_t^{(a)}[v_{B_1}]\otimes\bigotimes_{j=2}^qZ_t^{(1)}[v_{B_j}]\right|^2\right)^{1/2}  \\
&\qquad\le C(e^{Ct}-1)\|f\|_{\Cwa{n}}.
\end{aligned}
\]
\end{proof}

The next lemma justifies differentiation under the expectation sign.
\begin{lemma} \label{lem:diff_under_expectation}
Let \(n\in\N\), \(\alpha\in\R\), let Assumption~\ref{asm:classical} hold, and \(f\in\Cwa{n}(\R^d)\). Then \(P_tf\in\Cloc{n}(\R^d)\) and, for every \(r\in\{1,\dots,n\}\),
\[
\DD^r P_t f(x)[v_1,\dots,v_r]=\E_x\Big[
\sum_{k=1}^r \sum_{\pi=\{B_1,\dots,B_k\}\in \mathcal P(r,k)}
\DD^k f(X_t(x))
\big[
Y_t^{(|B_1|)}(x)[v_{B_1}],
\dots,
Y_t^{(|B_k|)}(x)[v_{B_k}]
\big]
\Big].
\]
\end{lemma}

\begin{proof}
Fix a compact \(K\subset\R^d\), \(r\le n\), a partition \(\pi=\{B_1,\dots,B_k\}\in\mathcal P(r,k)\), \(s_j:=|B_j|\), and unit vectors \(v_1,\dots,v_r\). With \(J_kf=w_{k-\alpha}\DD^kf\) and \(Z_t^{(s_j)}=w_{s_j-1}(X_t)Y_t^{(s_j)}\), the identity \(\sum_js_j=r\) gives
\[
\DD^kf(X_t)\big[Y_t^{(s_1)}[v_{B_1}],\dots,Y_t^{(s_k)}[v_{B_k}]\big]
=w_{\alpha-r}(x)\,\frac{w_{\alpha-r}(X_t)}{w_{\alpha-r}(x)}\,J_kf(X_t)\big[Z_t^{(s_1)}[v_{B_1}],\dots,Z_t^{(s_k)}[v_{B_k}]\big].
\]
With \(\|J_kf\|_\infty\le\|f\|_{\Cwa{n}}\) and \(w_{\alpha-r}\) bounded on \(K\), Lemma~\ref{lem:weighted_product_multi} yields
\[
\sup_{x\in K}\E_x\big|\DD^kf(X_t)[Y_t^{(s_1)}[v_{B_1}],\dots,Y_t^{(s_k)}[v_{B_k}]]\big|^p<\infty
\qquad\text{for some }p>1.
\]
Standard arguments using this local \(L^p\)-bound (uniform integrability and finite-difference approximation, cf.\ \cite[\S3.4]{Kunita2019}) give the displayed formula and continuity of \(\DD^rP_tf\).
\end{proof}

\section{Proofs of the main results} \label{sec:pde_appl}

The SDE derivative theorems follow directly from the estimates of Section~\ref{sec:derivative}; the Cauchy-problem proofs occupy the two subsections below, one per assumption.

\noindent\textbf{Proof of Theorem~\ref{thm:sde-A}.}

The case \(r=1\) is Proposition~\ref{lemm:multi_d_1st}; for \(r\ge2\), Lemma~\ref{lemm:multid} gives \eqref{eq:sde-A-bound}.

\noindent\textbf{Proof of Theorem~\ref{thm:sde-B}.}

The estimate at order one is Proposition~\ref{lemm:multi_d_1st}. For \(r\ge2\), Lemma~\ref{prop:multid_bounded_derivatives} gives \eqref{eq:sde-B-bound}.

\subsection{Cauchy problem under Assumption~\ref{asm:classical}}\label{sec:pde_appl:A}

\begin{lemma}\label{lem:weighted_density}
For every \(n\in\N_0\) and \(\alpha\in\R\), \(C_c^\infty(\R^d)\) is dense in \(\Cwainf{n}(\R^d)\) with respect to \(\|\cdot\|_{\Cwa{n}}\).
\end{lemma}

\begin{proof}
Let \(f\in\Cwainf{n}(\R^d)\), choose \(\chi\in C_c^\infty(\R^d)\) with \(\chi=1\) on \(B_1\) and \(\chi=0\) outside \(B_2\), and set \(\chi_R(x)=\chi(x/R)\), \(f_R=\chi_Rf\). For \(0\le j\le n\), Leibniz' rule writes \(\DD^j(f-f_R)\) as \((1-\chi_R)\DD^jf\) plus terms supported in \(\{R\le |x|\le2R\}\) containing \(\DD^\ell\chi_R\), \(1\le\ell\le j\). The first term tends to zero because \(\|\DD^jf\|/w_{\alpha-j}\in C_\infty(\R^d)\). On this region, \(\|\DD^\ell\chi_R\|_\infty\le C_\ell R^{-\ell}\) and \(R^{-\ell}\le Cw_{-\ell}(x)\), so each remaining term is bounded by
\[
    C\sup_{|x|\ge R}\frac{\|\DD^{j-\ell}f(x)\|}{w_{\alpha-j+\ell}(x)},
\]
which also tends to zero. Hence \(f_R\to f\) in \(\Cwa{n}\); mollifying \(f_R\) gives approximation by \(C_c^\infty\) functions in the same norm.
\end{proof}

\begin{lemma}\label{lem:weighted_compact_core}
Let \(N\in\N_0\), suppose Assumption~\ref{asm:classical} is satisfied for some integer \(n\ge N\), and let \(g\in C_c^\infty(\R^d)\). Then, for every compact \(K\subset\R^d\) and every \(0\le m\le N\),
\[
    \sup_{x\in K}\frac{\|\DD^mP_tg(x)-\DD^mg(x)\|}{w_{\alpha-m}(x)}\to0,
    \qquad t\downarrow0.
\]
\end{lemma}

\begin{proof}
On \(K\) the weights are bounded above and below, so it suffices to control \(\sup_{x\in K}\|\DD^mP_tg(x)-\DD^mg(x)\|\); the case \(m=0\) is the Feller property. Fix \(1\le m\le N\) and \(|v_i|\le1\), and split \(\DD^mP_tg[v]\) by Lemma~\ref{lem:diff_under_expectation} into the singleton part \(S_t(x)[v]=\E_x[\DD^mg(X_t)[Y_t^{(1)}v_1,\dots,Y_t^{(1)}v_m]]\) and the non-singleton remainder \(R_t(x)[v]\). Each summand of \(R_t\) carries a factor \(Y_t^{(a)}\) with \(a\ge2\); the H\"older inequality together with Proposition~\ref{lemm:multi_d_1st} and Lemma~\ref{lemm:multid} gives \(\sup_{x\in K}\|R_t(x)\|\le C_g(e^{ct}-1)^{1/p}e^{ct}\to0\) for some \(p>1\). For the singleton part, expanding \(\bigotimes_i(I+(Y_t^{(1)}-I))v_i\) gives
\[
    S_t(x)[v]-\DD^mg(x)[v]=\bigl(\E_x[\DD^mg(X_t)[v]]-\DD^mg(x)[v]\bigr)+B_t(x)[v],
\]
where \(B_t\) collects the \(2^m-1\) terms with at least one factor \(Y_t^{(1)}-I\). The first difference tends to zero uniformly on \(K\) because each component of \(\DD^mg\) lies in \(C_\infty(\R^d)\), while \(|B_t(x)[v]|\le C_g\E_x[(1+\|Y_t^{(1)}\|^{m-1})\|Y_t^{(1)}-I\|]\). Since \(Y_t^{(1)}-I\) starts from zero, the Burkholder--Davis--Gundy inequality and Proposition~\ref{lemm:multi_d_1st} give \(\sup_{x\in K}\E_x\|Y_t^{(1)}-I\|^p\le C_Kt^{p/2}\) for \(0<t\le1\), so \(B_t\to0\) and \(\sup_{x\in K}\|S_t(x)-\DD^mg(x)\|\to0\) as \(t\downarrow0\).
\end{proof}

\begin{proposition}\label{prop:weighted_semigroup_support}
Suppose Assumption~\ref{asm:classical} is satisfied for the integer \(n\), and let \(\alpha\in\R\). For every \(N\in\{0,\dots,n\}\), the semigroup maps \(\Cwainf{N}(\R^d)\) into itself, is strongly continuous there, and there is \(c_N>0\) such that
\[
    \|P_tf\|_{\Cwa{N}}\le e^{c_Nt}\|f\|_{\Cwa{N}},
    \qquad f\in\Cwainf{N}(\R^d),\quad t\ge0 .
\]
\end{proposition}

\begin{proof}
For \(N=0\) this is Lemma~\ref{lemm:invariant} and Proposition~\ref{theo:strongly_continuous}. Take \(N\ge1\); the zeroth-order estimate is Lemma~\ref{lemm:invariant}. For \(1\le r\le N\), Lemma~\ref{lem:diff_under_expectation} gives the Fa\`a di Bruno representation. Put \(J_qf=w_{q-\alpha}\DD^qf\) and \(Z_t^{(j)}=w_{j-1}(X_t)Y_t^{(j)}\). A partition \(\pi=\{B_1,\dots,B_q\}\in\mathcal P(r,q)\), \(s_i=|B_i|\), contributes
\[
w_{\alpha-r}(x)
\frac{w_{\alpha-r}(X_t)}{w_{\alpha-r}(x)}
J_qf(X_t)
\big[Z_t^{(s_1)}[v_{B_1}],\dots,Z_t^{(s_q)}[v_{B_q}]\big].
\]
For the singleton partition, \(q=r\) and \(s_i=1\). Writing \(J_rf=w_{r-\alpha}\DD^rf\), we have \(\|J_rf\|_\infty\le\|f\|_{\Cwa{N}}\); Cauchy--Schwarz and Case 1 of Lemma~\ref{lem:weighted_product_multi}, applied to \(R_t^{\alpha-r}\bigotimes_i Z_t^{(1)}[v_i]\), give an \(e^{ct}\|f\|_{\Cwa{N}}w_{\alpha-r}(x)\) bound. Every non-singleton partition, thanks to Proposition~\ref{prop:weighted_non_singleton_cancellation}, contributes at most \(C(e^{Ct}-1)\|f\|_{\Cwa{N}}w_{\alpha-r}(x)\). Summing gives the displayed estimate (with \(c=c_N\)).

To show that the semigroup is strongly continuous, let \(g\in C_c^\infty(\R^d)\) and \(\beta<\alpha\). The estimate just proved, applied with weight \(\beta\), gives \(\sup_{0\le t\le1}\|P_tg\|_{\Cwa[\beta]{N}}<\infty\). Hence, for each \(0\le k\le N\),
\[
    \frac{\|\DD^kP_tg(x)\|}{w_{\alpha-k}(x)}
    \le C_g w_{\beta-\alpha}(x)\to0
    \qquad (|x|\to\infty),
\]
uniformly for \(0\le t\le1\). Combining this estimate on the tail \(|x|\to\infty\) with the estimate on compact sets from Lemma~\ref{lem:weighted_compact_core}, we obtain \(\|P_tg-g\|_{\Cwa{N}}\to0\) and \(P_tg\in\Cwainf{N}\). For general \(f\in\Cwainf{N}\), choose \(g_m\in C_c^\infty\) with \(g_m\to f\) in \(\Cwa{N}\) by Lemma~\ref{lem:weighted_density}. The estimate shows \(P_tg_m\to P_tf\) in \(\Cwa{N}\) for fixed \(t\); since \(\Cwainf{N}\) is closed in \(\Cwa{N}\), \(P_tf\in\Cwainf{N}\). The  decomposition
\[
    \|P_tf-f\|_{\Cwa{N}}\le \|P_t(f-g_m)\|_{\Cwa{N}}+\|P_tg_m-g_m\|_{\Cwa{N}}+\|g_m-f\|_{\Cwa{N}},
\]
with \(m\) first large and then \(t\downarrow0\), proves strong continuity.
\end{proof}

\noindent\textbf{Proof of Theorem~\ref{thm:cauchy-A}.}

Let $A_{n-2}$ denote the generator of $(P_t)_{t\ge0}$ on $\Cwainf{n-2}(\R^d)$. Proposition~\ref{prop:weighted_semigroup_support} gives strong continuity of \((P_t)_{t\ge0}\) on both \(\Cwainf{n}(\R^d)\) and \(\Cwainf{n-2}(\R^d)\), and \(P_t(\Cwainf{n}(\R^d))\subset \Cwainf{n}(\R^d)\) for all \(t\ge0\).

\begin{lemma}\label{lem:weighted_generator}
Under Assumption \ref{asm:classical} with $n\ge2$,
\[
\Cwainf{n}(\R^d)\subset D(A_{n-2}),
\qquad
A_{n-2}\varphi=\cL\varphi
\quad\text{for every }\varphi\in \Cwainf{n}(\R^d).
\]
\end{lemma}

\begin{proof}
Let \(\varphi\in \Cwainf{n}(\R^d)\). From \(a=\sigma\sigma^\top\) and Assumption~\ref{asm:classical} (\(\|\DD^j\sigma\|\le Cw_{1-j}\)), the Leibniz rule gives
\[
    \|\DD^i a(x)\|
    \le \sum_{k+l=i}\tbinom{i}{k}\,\|\DD^k\sigma(x)\|\,\|\DD^l\sigma(x)\|
    \le C\sum_{k+l=i}w_{1-k}(x)\,w_{1-l}(x)
    \le C\,w_{2-i}(x),
    \qquad 0\le i\le n.
\]
For \(0\le j\le n-2\), Leibniz' rule expands
\[
    \DD^j(b\cdot\nabla\varphi)=\sum_{k+l=j}\tbinom{j}{k}\,\DD^k b\otimes\DD^{l+1}\varphi.
\]
Each summand, divided by \(w_{\alpha-j}\), factors as the bounded ratio \(\|\DD^k b\|/w_{1-k}\) (by Assumption~\ref{asm:classical}) times the ratio \(\|\DD^{l+1}\varphi\|/w_{\alpha-(l+1)}\), which lies in \(C_\infty\) because \(\varphi\in \Cwainf{n}(\R^d)\). The same Leibniz expansion for \(\DD^j(\operatorname{tr}(a\DD^2\varphi))\), with \(\|\DD^i a\|\le Cw_{2-i}\) above, gives the analogous bound. Hence \(\cL\varphi\in \Cwainf{n-2}(\R^d)\) with \(\|\cL\varphi\|_{\Cwa{n-2}}\le C\|\varphi\|_{\Cwa{n}}\). Lemma~\ref{lemm:Lyapunov} gives \(P_tw_\beta\le C_Tw_\beta\), so \(\varphi(X_t)\) and \(\int_0^t\cL\varphi(X_s)\,ds\) are integrable. Itô's formula on \(X_{t\wedge\tau_R}\) with \(\tau_R=\inf\{s:|X_s|\ge R\}\) and \(R\to\infty\) gives \(P_t\varphi-\varphi=\int_0^tP_s(\cL\varphi)\,ds\), and strong continuity on \(\Cwainf{n-2}(\R^d)\) yields \((P_t\varphi-\varphi)/t\to\cL\varphi\) in \(\Cwainf{n-2}(\R^d)\).
\end{proof}

Apply the lemma to \(\varphi=f\), and set
\[
u(t,x):=P_t f(x)=\E_x[f(X_t)].
\]
Since $P_t(\Cwainf{n}(\R^d))\subset \Cwainf{n}(\R^d)$, for every $t\ge0$ we have $u(t,\cdot)\in \Cwainf{n}(\R^d)$. Moreover,
\[
\frac{u(t+h,\cdot)-u(t,\cdot)}{h}=\frac{P_h P_t f-P_t f}{h},
\]
and by Lemma \ref{lem:weighted_generator},
\[
\frac{u(t+h,\cdot)-u(t,\cdot)}{h}\to \cL u(t,\cdot)
\qquad\text{in }\Cwainf{n-2}(\R^d).
\]
Hence
\[
u\in C^1([0,\infty);\Cwainf{n-2}(\R^d)),
\qquad
\partial_t u(t,\cdot)=\cL u(t,\cdot).
\]
Strong continuity of $(P_t)$ on $\Cwainf{n}(\R^d)$ gives
\[
u(t,\cdot)\to f \quad\text{in }\Cwainf{n}(\R^d) \qquad (t\downarrow0),
\]
so the initial condition holds. Hence $u\in C^{1,2}((0,\infty)\times\R^d)$, and $u$ is a classical solution.
Proposition~\ref{prop:weighted_semigroup_support} gives, for every \(t\ge0\),
\[
    \|u(t,\cdot)\|_{\Cwa{n}}=\|P_tf\|_{\Cwa{n}}\le e^{ct}\|f\|_{\Cwa{n}}.
\]
Therefore, for every \(k=0,\dots,n\),
\[
    \|\DD^ku(t,x)\|\le e^{ct}\|f\|_{\Cwa{n}}w_{\alpha-k}(x).
\]
The formula for the spatial derivatives follows from Lemma \ref{lem:diff_under_expectation}. In particular,
\[
\nabla u(t,x)[v]=\E_x\big[\nabla f(X_t(x))[Y_t^{(1)}(x)[v]]\big].
\]

Uniqueness follows from \cite[Theorem~5.7.6]{KaratzasShreve1998}.

\subsection{Cauchy problem under Assumption~\ref{asm:bounded_derivatives}}\label{sec:pde_appl:B}

\begin{lemma}\label{lem:diff_under_expectation_bounded_derivatives}
Let \(n\in\N\), let Assumption \ref{asm:bounded_derivatives} hold, and let \(f\in \Cwbd{n}(\R^d)\). Then \(P_tf\in \Cloc{n}(\R^d)\). Moreover, for every \(r=1,\dots,n\),
\[
\DD^rP_tf(x)[v_1,\dots,v_r]=\E_x\Big[
\sum_{q=1}^r\sum_{\pi=\{B_1,\dots,B_q\}\in\mathcal P(r,q)}
\DD^qf(X_t)
\big[
Y_t^{(|B_1|)}[v_{B_1}],\dots,Y_t^{(|B_q|)}[v_{B_q}]
\big]\Big].
\]
\end{lemma}
\begin{proof}
Since \(\Cwbd{n}\subset\Cwa[2]{0}(\R^d)\), zeroth-order continuity is Lemma~\ref{lemm:weighted_feller} with \(\alpha=2\). For \(r\le n\), Proposition~\ref{prop:higher_jet_sde} gives the \(C^n\)-spatial modification of the SDE solution. Fix a compact \(K\subset\R^d\), a partition \(\pi=\{B_1,\dots,B_q\}\), and unit \(v_i\). Boundedness of \(\DD^qf\) (since \(q\ge1\)), Proposition~\ref{lemm:multi_d_1st} on singleton blocks, and Lemma~\ref{prop:multid_bounded_derivatives} on the rest, combined with the H\"older inequality, give
\[
    \sup_{x\in K}\E_x\bigl|\DD^qf(X_t)[Y_t^{(|B_1|)}[v_{B_1}],\dots,Y_t^{(|B_q|)}[v_{B_q}]]\bigr|^p<\infty
\]
for some \(p>1\). The standard arguments of Lemma~\ref{lem:diff_under_expectation} transfer and give the needed formula and continuity of \(\DD^rP_tf\).
\end{proof}

\begin{lemma}\label{lem:bounded_derivative_density}
For every \(n\in\N_0\), \(C_c^\infty(\R^d)\) is dense in \(\Cwbdinf{n}(\R^d)\) with respect to \(\|\cdot\|_{\Cwbd{n}}\).
\end{lemma}

\begin{proof}
Let \(f\in\Cwbdinf{n}(\R^d)\) and use the cutoffs \(\chi_R\) from Lemma~\ref{lem:weighted_density}. By Leibniz' rule, for \(0\le j\le n\) the derivative \(\DD^j(f-\chi_Rf)\) equals \((1-\chi_R)\DD^jf\) plus cross terms, each carrying a factor \(\DD^\ell\chi_R\) with \(1\le\ell\le j\) and hence supported in the region \(\{R\le|x|\le2R\}\). The leading term \((1-\chi_R)\DD^jf\) tends to zero: for \(j=0\), \(\|(1-\chi_R)f\|_{\Cwa[2]{0}}\to0\) since \(f/w_2\in C_\infty(\R^d)\); for \(j\ge1\), \(\sup_{|x|\ge R}\|\DD^jf\|\to0\) since \(\DD^jf\in C_\infty(\R^d)\). For a cross term, \(\|\DD^\ell\chi_R\|_\infty\le C_\ell R^{-\ell}\): if \(\ell<j\), the factor \(\DD^{j-\ell}f\in C_\infty(\R^d)\) makes it vanish; the remaining term \(\ell=j\) leaves \(f\) undifferentiated. For \(\ell=j=1\), write \(x=\rho e\) with \(|e|=1\); the fundamental theorem of calculus along the ray \(s\mapsto se\) gives
\[
    |f(\rho e)|\le |f(0)|+\int_0^\rho|\nabla f(se)|\,ds,
\]
so \(\nabla f\in C_\infty\) yields \(|f(\rho e)|/\rho\to0\) as \(\rho\to\infty\), uniformly in \(e\), and therefore \(R^{-1}\sup_{R\le|x|\le2R}|f|\to0\). For \(\ell=j\ge2\), \(R^{-j}\sup_{R\le|x|\le2R}|f|\le C R^{2-j}\sup_{|x|\ge R}|f|/w_2\to0\). Thus \(\chi_Rf\to f\) in \(\Cwbd{n}\); mollifying \(\chi_Rf\) gives the claim.
\end{proof}

\begin{lemma}\label{lem:bounded_derivative_compact_core}
Let \(N\in\N_0\), and suppose Assumption~\ref{asm:bounded_derivatives} is satisfied for some integer \(n\ge N\). If \(g\in C_c^\infty(\R^d)\), then \(P_tg\in\Cwbdinf{N}(\R^d)\) for every \(t\ge0\), and
\[
    \|P_tg-g\|_{\Cwbd{N}}\to0,
    \qquad t\downarrow0.
\]
\end{lemma}

\begin{proof}
The case \(r=0\) is Proposition~\ref{theo:strongly_continuous} with \(\alpha=2\). For \(1\le r\le N\), the Fa\`a di Bruno formula of Lemma~\ref{lem:diff_under_expectation_bounded_derivatives} yields the singleton/non-singleton decomposition of the proof of Lemma~\ref{lem:weighted_compact_core}:
\[
    \DD^rP_tg-\DD^rg=R_t+\bigl(\E[\DD^rg(X_t)]-\DD^rg\bigr)+B_t,
\]
where \(B_t\) collects the \(2^r-1\) non-identity terms in \(\bigotimes_i(I+(Y_t^{(1)}-I))v_i\). Under Assumption~\ref{asm:bounded_derivatives}, the moments from Proposition~\ref{lemm:multi_d_1st} and Lemma~\ref{prop:multid_bounded_derivatives} are uniform in \(x\). The three bounds used in Lemma~\ref{lem:weighted_compact_core} transfer: the H\"older inequality bounds \(\|R_t\|_\infty\); \(\DD^rg\in C_\infty\) controls the continuity difference; and Burkholder--Davis--Gundy with Proposition~\ref{lemm:multi_d_1st} bounds \(|B_t(x)|\) by \(C_g\E_x[(1+\|Y_t^{(1)}\|^{r-1})\|Y_t^{(1)}-I\|]\). Together they give \(\sup_x\|\DD^rP_tg(x)-\DD^rg(x)\|\to0\).

For \(P_tg\in\Cwbdinf{N}\), the H\"older inequality dominates each Fa\`a di Bruno summand containing \(\DD^qg(X_t)\) by \(C_g(P_t\varphi(x))^\theta\), \(\theta\in(0,1)\), where \(\varphi\in C_c^\infty(\R^d)\) equals \(1\) on \(\operatorname{supp}\DD^qg\); \(P_t\varphi\in C_\infty(\R^d)\) by Lemma~\ref{lemm:weighted_feller}.
\end{proof}

\begin{proposition}\label{prop:bounded_derivative_semigroup_support}
Let Assumption~\ref{asm:bounded_derivatives} hold with \(n=2\). Then \(P_t\) maps \(\Cwbdinf{2}(\R^d)\) into itself, is strongly continuous there, and is strongly continuous on \(\Cwbdinf{0}(\R^d)\). Moreover, there are constants \(c,M\ge0\), such that for \(f\in\Cwbdinf{2}(\R^d)\),
\[
|P_tf(x)|\le e^{ct}\|f\|_{\Cwbd{2}}w_2(x),\qquad
\|\DD P_tf(x)\|\le e^{ct}\|f\|_{\Cwbd{2}},
\]
and
\[
    \|\DD^2P_tf(x)\|\le (1+Mt^{1/2})e^{ct}\|f\|_{\Cwbd{2}} .
\]
\end{proposition}

\begin{proof}
The zeroth-order estimate is Lemma~\ref{lemm:invariant} with \(\alpha=2\). For the first derivative, Lemma~\ref{lem:diff_under_expectation_bounded_derivatives} and Proposition~\ref{lemm:multi_d_1st} give
\[
    |\DD P_tf(x)[v]|\le \|\DD f\|_\infty\,\E_x|Y_t^{(1)}v|
    \le e^{ct}\|f\|_{\Cwbd{2}}|v|.
\]
For the second derivative,
\[
\DD^2P_tf(x)[v_1,v_2]
=\E_x\big[\DD^2f(X_t)[Y_t^{(1)}v_1,Y_t^{(1)}v_2]\big]
+\E_x\big[\DD f(X_t)[Y_t^{(2)}[v_1,v_2]]\big].
\]
The first term is bounded by \(e^{ct}\|f\|_{\Cwbd{2}}|v_1||v_2|\). The second is bounded, by Lemma~\ref{prop:multid_bounded_derivatives} with \(p=2\), by \(C(e^{Ct}-1)^{1/2}\|f\|_{\Cwbd{2}}|v_1||v_2|\). Since \(e^{Ct}-1\le Cte^{Ct}\), we have \((e^{Ct}-1)^{1/2}\le (Ct)^{1/2}e^{Ct/2}\); the stated estimate on the second derivative follows after enlarging \(c\) and \(M\).

Let \(g_m\in C_c^\infty(\R^d)\) approximate \(f\) in \(\Cwbd{2}\), using Lemma~\ref{lem:bounded_derivative_density}. Lemma~\ref{lem:bounded_derivative_compact_core} gives \(P_tg_m\in\Cwbdinf{2}\), while the estimates above give \(P_tg_m\to P_tf\) in \(\Cwbd{2}\) for each fixed \(t\). Since \(\Cwbdinf{2}\) is closed in \(\Cwbd{2}\), \(P_tf\in\Cwbdinf{2}\). The same estimates, together with the convergence \(P_tg_m\to g_m\), give \(\|P_tf-f\|_{\Cwbd{2}}\to0\). Strong continuity on \(\Cwbdinf{0}=\Cwainf[2]{0}\) follows from Proposition~\ref{theo:strongly_continuous} with \(\alpha=2\).
\end{proof}

Let \(A_0\) denote the generator of \((P_t)_{t\ge0}\) on \(\Cwbdinf{0}(\R^d)\).

\begin{lemma}\label{lem:bounded_derivatives_generator}
Under Assumption \ref{asm:bounded_derivatives} with \(n=2\),
\[
\Cwbdinf{2}(\R^d)\subset D(A_0),
\qquad
A_0\varphi=\cL\varphi
\quad\text{for every }\varphi\in \Cwbdinf{2}(\R^d).
\]
\end{lemma}

\begin{proof}
Let $\varphi\in \Cwbdinf{2}(\R^d)$. Assumption~\ref{asm:bounded_derivatives} with $n=2$ gives $|b|=O(w_1)$ and $\|a\|=O(w_2)$, while $\nabla\varphi,\DD^2\varphi\in C_\infty(\R^d)$; the Leibniz rule then yields $\cL\varphi\in \Cwbdinf{0}(\R^d)$, with $|\varphi|+|\cL\varphi|\le Cw_2$. By Lemma~\ref{lemm:Lyapunov}, we have the integrability needed for It\^o's formula on $X_{t\wedge\tau_R}$ with $\tau_R=\inf\{s:|X_s|\ge R\}$. Letting $R\to\infty$ yields $P_t\varphi-\varphi=\int_0^t P_s(\cL\varphi)\,ds$. Strong continuity on $\Cwbdinf{0}(\R^d)=\Cwainf[2]{0}(\R^d)$ (by Proposition~\ref{theo:strongly_continuous} with $\alpha=2$) then gives $(P_t\varphi-\varphi)/t\to\cL\varphi$ in $\Cwbdinf{0}(\R^d)$, so $\varphi\in D(A_0)$ and $A_0\varphi=\cL\varphi$.
\end{proof}

\noindent\textbf{Proof of Theorem~\ref{thm:cauchy-B}.}

Set \(u(t,\cdot):=P_tf\). Proposition~\ref{prop:bounded_derivative_semigroup_support} gives \(u\in C([0,\infty);\Cwbdinf{2}(\R^d))\), \(P_t(\Cwbdinf{2}(\R^d))\subset \Cwbdinf{2}(\R^d)\), strong continuity on \(\Cwbdinf{0}(\R^d)\), and the three displayed estimates in Theorem~\ref{thm:cauchy-B}. By Lemma~\ref{lem:bounded_derivatives_generator}, \(\Cwbdinf{2}(\R^d)\subset D(A_0)\) and \(A_0\varphi=\cL\varphi\) on \(\Cwbdinf{2}(\R^d)\). Hence semigroup differentiability on the generator domain gives
\[
    \frac{P_{t+h}f-P_tf}{h}=\frac{P_h(P_tf)-P_tf}{h}\to A_0P_tf=\cL P_tf
    \qquad\text{in }\Cwbdinf{0}(\R^d),
\]
so \(u\in C^1([0,\infty);\Cwbdinf{0}(\R^d))\) and \(\partial_tu=\cL u\); Lemma~\ref{lem:diff_under_expectation_bounded_derivatives} gives the spatial smoothness, hence \(u\) is a classical solution. Any solution in the stated class lies in \(\Cwbdinf{2}(\R^d)\) for each \(t\), hence has \(w_2\)-polynomial growth, uniformly on compact time intervals. Uniqueness then follows from  \cite[Theorem~5.7.6]{KaratzasShreve1998}.

\section*{Declarations}


The research was supported by Vega Institute Foundation.





\bibliography{sn-bibliography}

\end{document}